\documentclass{article}

\usepackage{arxiv}

\usepackage[utf8]{inputenc} 
\usepackage[T1]{fontenc}    
\usepackage{hyperref}       
\usepackage{url}            
\usepackage{booktabs}       
\usepackage{amsfonts}       
\usepackage{nicefrac}       
\usepackage{microtype}      
\usepackage{graphicx}
\usepackage{natbib}
\usepackage{doi}
\usepackage{amsthm}
\usepackage{amsmath}
\usepackage{amssymb}
\usepackage{amsfonts}
\usepackage{csquotes}
\usepackage{bm}
\usepackage{tcolorbox}
\usepackage{calc,pifont}
\usepackage{caption}
\usepackage{mathrsfs}
\usepackage{tikz}
\usepackage{bbm}
\usepackage{xcolor,colortbl}

\usepackage{hyperref}
\usepackage[noabbrev]{cleveref} 

\newtheorem{theorem}{Theorem}[section] 

\newtheorem{lemma}[theorem]{Lemma} 
\newtheorem{definition}[theorem]{Definition} 
\newtheorem{example}[theorem]{Example} 
\newtheorem{proposition}[theorem]{Proposition} 
\newtheorem*{remark}{Remark}

\title{Phase Reduction of Limit Cycle Oscillators: A Tutorial Review with New Perspectives on Isochrons and an Outlook to Higher-Order Reductions}

\author{ \href{https://orcid.org/0000-0002-0508-152X}{\includegraphics[scale=0.06]{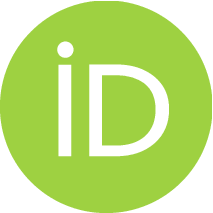}\hspace{1mm}Zeray Hagos Gebrezabher} \\
	Faculty of Mathematics and Statistics\\
	Mekelle University\\
	Mekelle, Ethiopia \\
	\texttt{zeray.hagos.geb@gmail.com} 
}

\date{}

\hypersetup{
pdftitle={Phase Reduction of Limit Cycle Oscillators: A Tutorial Review with New Perspectives on Isochrons and an Outlook to Higher-Order Reductions},
pdfsubject={math},
pdfauthor={Zeray Hagos Gebrezabher},
pdfkeywords={Phase reduction, Limit cycle oscillators, Isochrons, Asymptotic phase, Synchronization, Coupled oscillators, Higher-order reductions},
}

\begin{document}
\maketitle

\begin{abstract}
The phase reduction technique is essential for studying rhythmic phenomena across various scientific fields. It allows the complex dynamics of high-dimensional oscillatory systems to be expressed by a single phase variable. This paper provides a detailed review and synthesis of phase reduction with two main goals. First, we develop a solid geometric framework for the theory by creating isochrons, which are the level sets of the asymptotic phase, using the Graph Transform theorem. We show that isochrons form an invariant, continuous structure of the basin of attraction of a stable limit cycle, helping to clarify the concept of the asymptotic phase. Second, we systematically explain how to derive the first-order phase reduction for weakly perturbed and coupled systems. In the end, we discuss the limitations of the first-order approach, particularly its restriction to very small perturbations and the issue of vanishing coupling terms in certain networks. We finish by outlining the framework and importance of higher-order phase reductions. This establishes a clear link from classical theory to modern developments and sets the stage for a more in-depth discussion in a future publication. 
\end{abstract}


\section{Introduction}
\label{sec:introduction}

The coordinated firing of cardiac pacemaker cells \citep{Leon1988Machey, Winfree2001} and the synchronised flashing of fireflies \citep{Newman2006} exemplify the ubiquitous presence of stable rhythmic behaviour in biological and physical systems.  These self-sustained oscillations are mathematically represented by stable limit cycles in nonlinear dynamical systems.  Despite the complexity of the underlying dynamics, the long-term behaviour near the cycle can frequently be characterised by a single phase variable.  The robust \emph{phase reduction} technique utilises this understanding to simplify oscillatory models, thereby enhancing the examination of synchronisation, entrainment, and stimulus response \citep{kuramoto1984chemical, nakao2016phase, pikovsky2003synchronization, Pietras2019}.

The theoretical underpinnings of phase reduction were established by \cite{winfree1967biological} in his exploration of biological rhythms and \cite{kuramoto1984chemical} regarding chemical oscillations.  The primary elements of this theory are \emph{isochrons}, which are submanifolds of points within the basin of attraction that possess identical asymptotic phases \cite{guckenheimer1975isochrons}.  The gradient of the asymptotic phase function, referred to as the Phase Response Curve (PRC) or sensitivity function \(Z(\theta)\), measures the oscillator's sensitivity to perturbations and comprehensively delineates its dynamics under weak forcing \citep{winfree1967biological, ermentrout1996type}.

Despite its prevalent application, a comprehensive analysis of the existence and structure of isochrons is frequently confined to specialised mathematical literature.  In numerous practical scenarios, the presence of a smooth asymptotic phase function is presumed without exploring the underlying geometry that facilitates it.  Moreover, the conventional first-order phase reduction is applicable solely for negligible perturbations.  For systems with moderate coupling strength or specific resonance conditions, higher-order corrections are essential to capture the true dynamics of the network \citep{Leon2019PhyReviewE, gengel2021high, ashwin2016phase, Kralemann2011reconstruction, Wilson2016Moehlis, Daido1992ProgressivePhys, Rosenblum2007Arkady, Kurebayashi2013Nakao, Pyragas2015PhysRevE, eddie2022NatureComm}, yet their derivation is often computationally intensive and not widely demonstrated.

This paper addresses important gaps by providing a tutorial review of phase reduction that emphasizes geometric understanding and precise mathematics. We assume readers have a basic grasp of ordinary differential equations, stability theory, and smooth manifolds. We move from basic ideas to more advanced topics like invariant foliations and normal hyperbolicity. Instead of an all-encompassing overview, we present an independent analysis based on a specific geometric framework. This perspective builds on the research from my doctoral dissertation \cite{gebrezabher2023thesis}, where I formulated and applied these principles to network reconstruction. Our contributions are threefold:

First, we provide a solid foundation by using the Graph Transform theorem to prove the existence and properties of isochrons for an exponentially stable limit cycle. This establishes a strong geometric basis and clearly shows that isochrons create an invariant foliation of the basin of attraction. Second, we present a systematic derivation of the first-order phase reduction for weakly perturbed systems and coupled oscillator networks, highlighting the importance of the phase sensitivity function \(Z(\theta)\) and the concept of averaging. This forms the theoretical basis for the network reconstruction methods outlined in \citep{gebrezabher2023thesis}. Finally, we critically assess the limitations of the first-order reduction to encourage future research. We demonstrate that its effectiveness is restricted to very small perturbations and fails for systems with certain symmetries or resonances that cancel out leading-order interactions. This discussion seamlessly leads into the ideas and challenges of higher-order phase reductions, setting the stage for a detailed exploration in a future publication.

The structure of the paper is as follows.  The required mathematical preliminaries are covered in Section \ref{sec:preliminaries}.  We describe our geometric construction of isochrons and the asymptotic phase in Section \ref{sec:geometry-asymphase}.  This is applied to first-order phase reduction in Section \ref{sec:phase-reduction}, and to networks in Section \ref{sec:applications}.  We address the drawbacks of the first-order method and present the idea of higher-order reductions in Section \ref{sec:limitations}.  In Section \ref{sec:conclusion}, we wrap up with a prediction.

\section{Mathematical Preliminaries}\label{sec:preliminaries}

To ensure self-containedness, we state some fundamental results that will be used throughout this paper. The systematic treatment of these mathematical foundations follows the approach developed in \citep{gebrezabher2023thesis}, which provides additional context and applications.

\begin{theorem}[\textbf{Picard-Lindelöf}]
Let \(D\subset\mathbb{R}\times\mathbb{R}^{n}\) be a closed rectangle with \((t_{0},x_{0})\in \text{int}(D)\). Let \(f:D\rightarrow\mathbb{R}^{n}\) be a function that is continuous in \(t\) and Lipschitz continuous in \(y\). Then, there exists some \(\varepsilon>0\) such that the initial value problem
\[
\dot{x}=f(t,x),\quad x(t_{0})=x_{0}
\]
has a unique solution \(x(t)\) on the interval \([t_{0}-\varepsilon,t_{0}+\varepsilon]\).
\end{theorem}
\begin{proof}
See \cite[Theorem 5.1]{hartman2002ordinary}.
\end{proof}
\begin{remark}
This theorem guarantees the existence and uniqueness of the flow \(\varphi(t, x)\), which is the fundamental object for defining the asymptotic phase and isochrons.
\end{remark}

\begin{lemma}[\textbf{Gronwall's Inequality}]\label{lem:Gronwall_lemma}
Consider \(U\subset\mathbb{R}^{+}\) and let \(u:U\rightarrow\mathbb{R}\) be a continuous and nonnegative function. Suppose there exist \(C\geq 0\) and \(K\geq 0\) such that
\[
u(t)\leq C+\int_{0}^{t}Ku(s)\,ds
\]
for all \(t\in U\), then
\[
u(t)\leq Ce^{Kt}.
\]
\end{lemma}
\begin{proof}
See \cite[Lemma 2.2]{hartman2002ordinary}.
\end{proof}
\begin{remark}
Gronwall's inequality is instrumental in proving the exponential stability of the limit cycle and the properties of the isochron foliation, particularly in demonstrating the contraction properties underlying the Graph Transform method.
\end{remark}

\begin{theorem}[Normal Hyperbolicity \citep{hirsch2012differential, fenichel1979geometric}]\label{thm:normal_hyperbolicity}
Let \(M\) be a compact, connected, invariant manifold for the flow \(\varphi(t,x)\) generated by \(\dot{x} = f(x)\), where \(f \in C^1\). The manifold \(M\) is \textbf{normally hyperbolic} if the linearized flow \(D\varphi(t,x)\) splits into continuous stable, unstable, and center bundles:
\[
T_x\mathbb{R}^n = E^s_x \oplus E^u_x \oplus T_xM,
\]
with the following growth conditions:
\begin{itemize}
    \item There exist constants \(C > 0\), \(\lambda_s < 0 < \lambda_u\) such that for all \(x \in M\):
    \begin{align*}
    \|D\varphi(t,x)v\| &\leq Ce^{\lambda_s t}\|v\|, \quad \text{for } v \in E^s_x, \, t \geq 0, \\
    \|D\varphi(t,x)v\| &\leq Ce^{\lambda_u t}\|v\|, \quad \text{for } v \in E^u_x, \, t \leq 0.
    \end{align*}
    \item The flow on \(M\) is dominated by the normal dynamics: \(\max(\lambda_s, -\lambda_u) < \lambda_c\), where \(\lambda_c\) characterizes the weakest contraction/expansion along \(M\).
\end{itemize}
If \(M\) is normally hyperbolic, then it persists under \(C^1\)-small perturbations and has stable and unstable manifolds \(W^s(M)\), \(W^u(M)\) that are \(C^1\) close to those of the unperturbed system.
\end{theorem}
\begin{proof}
See \cite{hirsch2012differential} and \cite{fenichel1979geometric} for the complete proof using the Graph Transform method.
\end{proof}
\begin{remark}
A stable limit cycle is a special case of a normally hyperbolic manifold where \(E^u_x = \{0\}\) (no unstable directions). The isochrons correspond to the stable manifolds of points on the limit cycle, and normal hyperbolicity ensures their smooth persistence under small perturbations. This provides the theoretical foundation for the phase reduction approach.
\end{remark}

\begin{theorem}[Periodic Averaging \citep{sanders1985averaging, sanders2007averaging}]\label{thm:periodic-averaging}
Consider a perturbed system of the form
\[
\dot{x}=\epsilon f(x,t,\epsilon);\quad x\in U\subset\mathbb{R}^{n},\quad 0\leq \epsilon \ll 1,
\]
where \(f:\mathbb{R}^{n}\times\mathbb{R}\times\mathbb{R}^{+}\rightarrow\mathbb{R}^{n}\) is \(C^{k},k\geq 2\), bounded on bounded set \(U\), and of period \(T>0\) in \(t\). The associated \enquote{autonomous averaged system} is defined as
\[
\dot{y}=\epsilon\frac{1}{T}\int_{0}^{T}f(y,t,0)\,dt=:\epsilon\bar{f}(y).
\]
Then there exists a \(C^{k}\) change of coordinates \(x=y+\epsilon w(y,t,\epsilon)\) under which the system becomes
\[
\dot{y}=\epsilon\bar{f}(y)+\epsilon^{2}f_{1}(y,t,\epsilon),
\]
where \(f_{1}\) is of period \(T\) in \(t\). Moreover
\begin{itemize}
    \item[(i)] If \(x(t)\) and \(y(t)\) are solutions of the original and averaged systems starting at \(x_{0}\) and \(y_{0}\), respectively, at \(t=0\), and \(|x_{0}-y_{0}|=O(\epsilon)\), then \(|x(t)-y(t)|=O(\epsilon)\) on a time scale \(t\sim 1/\epsilon\).
    \item[(ii)] If \(p_{0}\) is a hyperbolic fixed point of the averaged system then there exists \(\epsilon_{0}>0\) such that, for all \(0<\epsilon\leq\epsilon_{0}\), the original system possesses a unique hyperbolic periodic orbit \(\gamma_{\epsilon}(t)=p_{0}+O(\epsilon)\) of the same stability type as \(p_{0}\).
\end{itemize}
\end{theorem}
\begin{proof}
See \citep{sanders1985averaging}.
\end{proof}
\begin{remark}
The averaging theorem is the workhorse for deriving the averaged phase equations for weakly forced and coupled oscillators, as shown in Sections \ref{sec:phase-reduction} and \ref{sec:applications}. It justifies the separation of fast and slow phase dynamics.
\end{remark}

\subsection{Averaging for Almost-Periodic Vector Fields}

The classical averaging theorem applies to periodic forcing, but many physical systems exhibit quasi-periodic or almost-periodic behavior. We extend the averaging framework using Bohr's theory of almost-periodic functions.

\begin{definition}[Almost-Periodic Function \cite{bohr1947almost}]\label{def:almost_periodic}
A continuous function \(f: \mathbb{R} \to \mathbb{R}^n\) is called \textbf{almost-periodic} if for every \(\epsilon > 0\), there exists a relatively dense set \(P_\epsilon \subset \mathbb{R}\) (called \(\epsilon\)-almost periods) such that
\[
\|f(t + \tau) - f(t)\| < \epsilon \quad \text{for all } t \in \mathbb{R} \text{ and } \tau \in P_\epsilon.
\]
Equivalently, \(f\) is almost-periodic if the family of translates \(\{f(\cdot + \tau) : \tau \in \mathbb{R}\}\) is precompact in the space of bounded continuous functions equipped with the supremum norm.
\end{definition}

\begin{definition}[Mean Value \cite{bohr1947almost}]\label{def:mean_value}
For an almost-periodic function \(f\), the \textbf{mean value} exists and is defined by
\[
M[f] = \lim_{T \to \infty} \frac{1}{T} \int_{t_0}^{t_0+T} f(t)\,dt,
\]
which is independent of \(t_0 \in \mathbb{R}\).
\end{definition}

\begin{definition}[Frequency Module \cite{bohr1947almost}]\label{def:frequency_module}
The \textbf{frequency module} \(\text{Mod}(f)\) of an almost-periodic function \(f\) is the set of all real numbers that can be approximated by integer linear combinations of the Fourier exponents of \(f\). If \(f\) has Fourier series
\[
f(t) \sim \sum_{k=1}^\infty a_k e^{i\lambda_k t},
\]
then \(\text{Mod}(f)\) is the smallest additive subgroup of \(\mathbb{R}\) containing all \(\lambda_k\).
\end{definition}

\begin{lemma}[Averaging Lemma for Almost-Periodic Systems \cite{sanders2007averaging, bogoliubov1961asymptotic}]\label{lem:almost_periodic_avg}
Consider the system
\[
\dot{x} = \epsilon f(x,t,\epsilon), \quad x \in \mathbb{R}^n, \quad 0 < \epsilon \ll 1,
\]
where \(f\) is almost-periodic in \(t\) uniformly in \(x\) on compact sets, and \(f \in C^k\), \(k \geq 2\). Define the averaged vector field
\[
\bar{f}(x) = M[f(x,\cdot,0)] = \lim_{T \to \infty} \frac{1}{T} \int_0^T f(x,t,0)\,dt.
\]
Then there exists an almost-periodic near-identity transformation \(x = y + \epsilon w(y,t,\epsilon)\) such that the system becomes
\[
\dot{y} = \epsilon \bar{f}(y) + \epsilon^2 f_1(y,t,\epsilon),
\]
where \(f_1\) is almost-periodic in \(t\). Moreover, if \(x(t)\) and \(y(t)\) are solutions of the original and averaged systems with \(|x(0) - y(0)| = O(\epsilon)\), then \(|x(t) - y(t)| = O(\epsilon)\) on time scales \(t \sim 1/\epsilon\).
\end{lemma}
\begin{proof}
The proof follows from constructing the transformation via the homological equation and using the properties of almost-periodic functions. See \citep{sanders2007averaging} and \citep{bogoliubov1961asymptotic} for details.
\end{proof}
\begin{remark}
This extension is crucial for phase reduction in systems with quasi-periodic forcing or coupling, which commonly occur in multi-frequency oscillatory networks. The almost-periodic averaging ensures that our phase reduction framework applies not only to simple periodic perturbations but also to more complex, multi-frequency interactions.
\end{remark}

\begin{theorem}[\textbf{Implicit Function Theorem}]
Let \(f:\mathbb{R}^{n+m}\rightarrow\mathbb{R}^{m}\) be a continuously differentiable function, and let \(\mathbb{R}^{n+m}\) have coordinates \((\boldsymbol{x},\boldsymbol{y})\). Fix a point \((\boldsymbol{a},\boldsymbol{b})\) with \(f(\boldsymbol{a},\boldsymbol{b})=\boldsymbol{0}\). If the Jacobian matrix \(J_{f,y}(\boldsymbol{a},\boldsymbol{b})\) is invertible, then there exists an open set \(U\subset\mathbb{R}^{n}\) containing \(\boldsymbol{a}\) such that there exists a continuously differentiable function \(g:U\rightarrow\mathbb{R}^{m}\) such that \(g(\boldsymbol{a})=\boldsymbol{b}\), and \(f(\boldsymbol{x},g(\boldsymbol{x}))=\boldsymbol{0}\) for all \(\boldsymbol{x}\in U\).
\end{theorem}
\begin{proof}
The proof of this theorem can be found in any standard book in analysis; for example, see \cite[Theorem 2.3]{lang1993implicit}.
\end{proof}
\begin{remark}
This theorem ensures the persistence of the limit cycle and the Poincaré map under small perturbations, a key step in establishing the robustness of the phase reduction framework.
\end{remark}

\section{The Geometry of Asymptotic Phase and Isochrons}\label{sec:geometry-asymphase}

\subsection{Dynamics Near a Periodic Orbit}

Consider a system of ordinary differential equations
\begin{equation}\label{eq:dyn_system}
\frac{dx}{dt}=f(x), \quad x \in U \subset \mathbb{R}^n,
\end{equation}
where \(f: U \rightarrow \mathbb{R}^{n}\) is continuously differentiable. By the Picard-Lindelöf theorem, for every initial condition \(x \in U\), there exists a unique continuously differentiable solution \(\varphi(t,x)\) satisfying
\[
\varphi(0,x)=x, \quad \frac{d\varphi(t,x)}{dt}=f(\varphi(t,x)).
\]

We can construct a \textbf{local flow} \cite{guckenheimer1983nonlinear}. Consider the set
\[
M:=\bigcup_{x\in U} I_x \times \{x\} \subset \mathbb{R} \times U.
\]
Then we can define the local flow \(\varphi: M \to U\) such that
\begin{enumerate}
    \item \(\varphi(0,x)=x\),
    \item For all \(x\in U\) and \(s\in I_x\), \(\varphi(t+s,x)=\varphi(t,\varphi(s,x))\) for all \(t\in I_x-s\).
\end{enumerate}
Condition (2) is called the \emph{group property} of the flow.

We assume that the system has an \emph{exponentially stable limit cycle} \(\gamma\) with period 1, i.e., \(\gamma(t+1)=\gamma(t)\) for all \(t\in\mathbb{R}\), and there exist constants \(K>0\), \(\lambda>0\) such that for each initial point \(x\) sufficiently close to \(\gamma\),
\[
d(\varphi(t,x),\gamma)\leq Ke^{-\lambda t},
\]
where \(d(x,\gamma)=\inf_{x_{0}\in\gamma}\|x-x_{0}\|\). The \textbf{basin of attraction} of \(\gamma\) is defined as
\[
\mathcal{B}(\gamma)=\{x\in U \mid \lim_{t\rightarrow\infty}d(\varphi(t,x),\gamma)=0\}.
\]

\begin{proposition}[\cite{guckenheimer1983nonlinear}]
The basin of attraction \(\mathcal{B}(\gamma)\) is an open set.
\end{proposition}

\begin{proof}
For every point \(x_0 \in \mathcal{B}(\gamma)\), there exists \(\epsilon > 0\) such that the ball \(B_\epsilon(x_0)\) is contained in \(\mathcal{B}(\gamma)\). Since \(\gamma\) is exponentially stable, there exists an open set \(N\) around \(\gamma\) such that any \(x \in N\) approaches \(\gamma\) exponentially fast. The point \(x_0 \in \mathcal{B}(\gamma)\) enters \(N\) in finite time \(t_0\), i.e., \(\varphi(t_0,x_0) \in N\). Since \(N\) is open, there exists an open neighborhood \(V = B_{\epsilon_0}(\varphi(t_0,x_0))\) contained in \(N\).

Define \(\alpha(t) := |\varphi(t,x_0) - \varphi(t,y)|\) for arbitrary \(y \in \mathcal{B}(\gamma)\). Then
\[
\alpha(t) \leq |x_0 - y| + K\int_0^t \alpha(s) ds.
\]
By Gronwall's inequality, \(|\varphi(t,x_0) - \varphi(t,y)| < |x_0 - y|e^{Kt}\). Choosing \(\epsilon = \epsilon_0 e^{-t_0 K}\) ensures that \(y \in B_\epsilon(x_0)\), proving the openness of \(\mathcal{B}(\gamma)\).
\end{proof}

\subsection{Phase Parameterization and Asymptotic Phase}

We can introduce a phase parameterization on the limit cycle \(\gamma\) in terms of \(\theta\in [0,1]\), where \([0,1]/\sim\) is homeomorphic to \(\mathbb{S}^1\). For points \(y\in\mathcal{B}(\gamma)\), we say \(y\) is in \textbf{asymptotic phase} with \(x\in\gamma\) if
\[
\lim_{t\to\infty}\|\varphi(t,y)-\varphi(t,x)\|=0.
\]
This leads to the definition of the \textbf{asymptotic phase function} \(\Theta:\mathcal{B}(\gamma)\to\mathbb{S}^1\) such that
\[
\lim_{t\to\infty}\|\varphi(t,y)-\varphi(t+\Theta(y),x)\|=0, \quad x\in\gamma.
\]
The schematic demonstration of the asymptotic phase function in the plane ($n=2$) is depicted in Fig.\ref{fig:asymptotic_phase}(a).

\begin{figure}[!htb]
\centering
\includegraphics[width=0.9\textwidth]{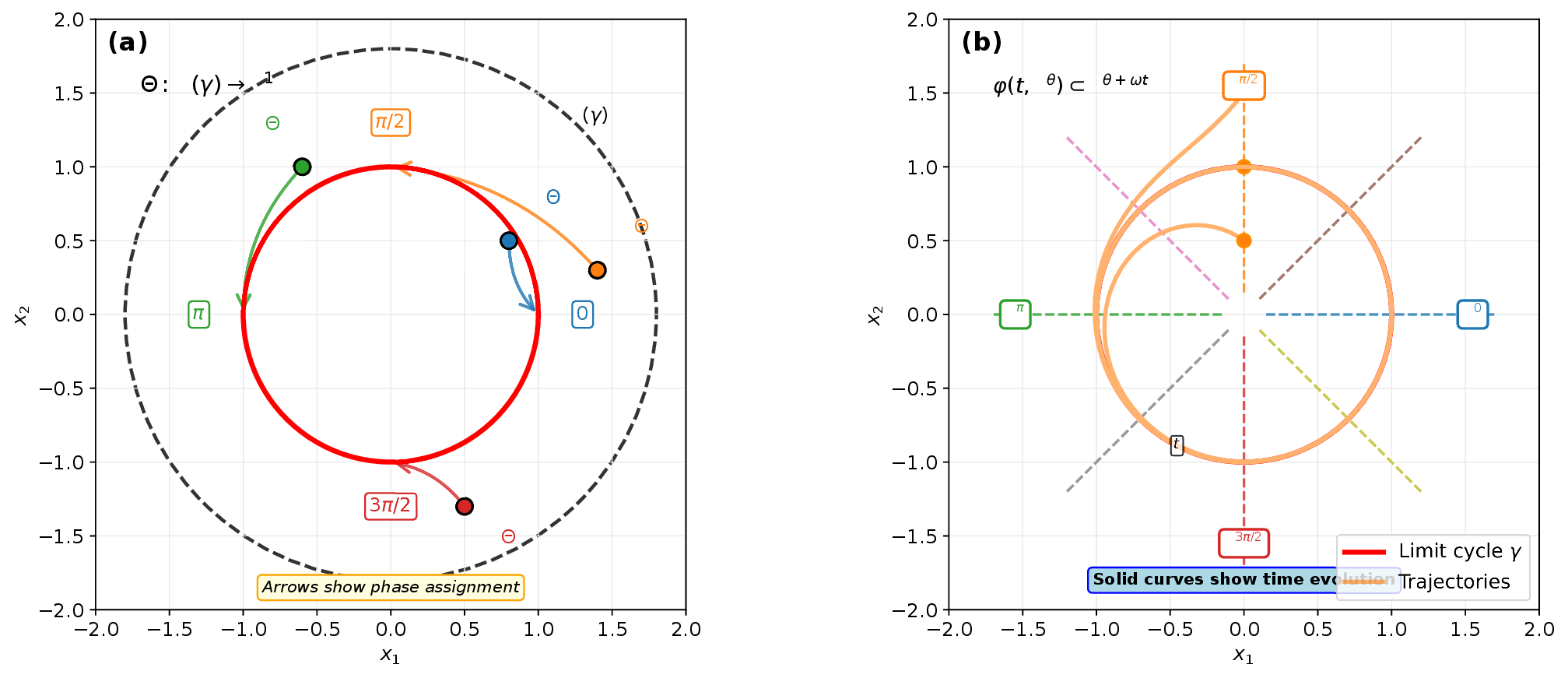}
\caption{\textbf{The asymptotic phase and isochrons for a stable limit cycle ($n=2$).} (a) Schematic illustration of the asymptotic phase function $\Theta:\mathcal{B}(\gamma)\to\mathbb{S}^1$. This function maps every point $x_0$ in the basin of attraction $\mathcal{B}(\gamma)$ to a phase $\theta\in\mathbb{S}^1$ on the limit cycle $\gamma$. This phase $\theta$ identifies the unique point on $\gamma$ that $x_0$ will synchronize with asymptotically. An isochron $\mathcal{I}^{\theta}(\gamma)$ (the blue surface) is the level set of the phase function, consisting of all points in the basin that share the same asymptotic phase $\theta$. The flow $\varphi(t,x)$ maps isochrons to isochrons, demonstrating the invariant foliation of the basin. That is, if two points start on the same isochron, they remain on the same isochron (and thus in the same phase relationship) for all time as they converge to $\gamma$. (b) A concrete example of radial isochrons for the system in \Cref{ex:simple_isochrons}. The limit cycle is the unit circle (red). The isochrons (dotted lines) are the radial lines of constant angle $\phi$, demonstrating that in this special case, the asymptotic phase is simply the angular coordinate.}
\label{fig:asymptotic_phase}
\end{figure}

\subsection{Isochrons and Invariant Foliation}

\begin{definition}[\cite{guckenheimer1975isochrons, winfree1967biological}]
The set of points \(x\) in \(\mathcal{B}(\gamma)\) with constant asymptotic phase \(\Theta(x)=\theta\) are called \textbf{isochrons}, denoted by \(\mathcal{I}^{\theta}\).
\end{definition}

Isochrons are level sets of the asymptotic phase function. The set of points where isochrons cannot be defined is called a \emph{phaseless set} \cite{guckenheimer1975isochrons}. See Figure \ref{fig:asymptotic_phase}(b) for a schematic illustration for $n=2$.

To characterize isochrons, we review the Poincaré map \cite{guckenheimer1983nonlinear}. A set \(\Sigma \subset \mathbb{R}^n\) is called a \emph{submanifold of codimension one} if it can be written as
\[
\Sigma = \{x \in U \mid s(x) = 0\},
\]
where \(U \subset \mathbb{R}^n\) is open, \(s \in C^k(U,\mathbb{R})\), \(k \geq 1\), and \(\nabla s(x) \neq 0\) for each \(x \in \Sigma\). The submanifold \(\Sigma\) is \emph{transversal} to \(f\) if \(\nabla s(x) \cdot f(x) \neq 0\) for each \(x \in \Sigma\).

Take a point \(x_0 \in \gamma\) and a transversal section \(\Sigma\) such that \(x_0 \in \Sigma\). The \textbf{Poincaré map} on \(\Sigma\) is defined as \(P: \Sigma \to \Sigma\) by
\[
P(x) := \varphi(\tau(x), x),
\]
where \(\tau: \Sigma \to \mathbb{R}\) is the \emph{time of first return}. If \(x_0 \in \gamma\), then \(\tau(x) \to 1\) as \(x \to x_0\), and \(x_0\) is a fixed point of \(P\).

Define the time-1 map \(g := \varphi(1, \cdot): \mathbb{R}^n \to \mathbb{R}^n\). By construction, \(g(\gamma) = \gamma\).

\begin{theorem}[Existence and Foliation of Isochrons]\label{thm:isochrons}
Consider the system \(\dot{x}=f(x)\) with \(f\in C^{k}(\mathbb{R}^{n})\), \(k\geq 1\). Assume \(\gamma\) is an exponentially stable limit cycle with basin of attraction \(\mathcal{B}(\gamma)\). For \(\epsilon\) small enough, through each \(x_{0}\in\gamma\), there is a unique isochron
\[
\mathcal{I}_{loc}(x_{0}):=\{z\in B_{\epsilon}(x_{0}) : |g^{m}(z)-x_{0}|<Ce^{-\lambda m},\,\lambda>0\},
\]
which is a graph of a \(C^{k}\) function. Moreover:
\begin{itemize}
    \item The union of isochrons forms an invariant foliation of a neighborhood \(V\) of \(\gamma\):
    \[
    V=\bigcup_{x_{0}\in\gamma}\mathcal{I}_{loc}(x_{0}), \quad \mathcal{I}_{loc}(x_{0})\cap\mathcal{I}_{loc}(y_{0})=\emptyset \text{ for } y_{0}\neq x_{0}\in\gamma.
    \]
    \item The foliation is invariant: for any \(x_{0}\in\gamma\) and \(t>0\),
    \[
    \varphi(t,\mathcal{I}_{loc}(x_{0}))\subset\mathcal{I}_{loc}(\varphi(t,x_{0})).
    \]
\end{itemize}
\end{theorem}
The proof constructs the isochron as a graph over the stable directions by showing that the graph transform is a contraction mapping on a suitable space of Lipschitz functions.
\begin{proof}
The proof proceeds using the \textbf{Graph Transform} method \cite{hirsch2012differential}. In this context, we only proof it for the 2-dimensional phase space, but which can be extended to the whole space.  Let \(g=\varphi(1,\cdot)\) be the time-1 map. For each \(x_{0}\in\gamma\), introduce local coordinates \(z=(u,v)\) near \(x_{0}\) (see Figure \ref{fig:coordinate_transform}a) such that \(x_{0}\) corresponds to the origin. In these coordinates, the time-1 map takes the form:
\[
\bar{u}=Au+R(u,v), \quad \bar{v}=v+S(u,v),
\]
where \(A\in(-1,1)\), and \(R,S\) are \(C^{1}\) with \(R(0,0)=S(0,0)=0\), \(DR(0,0)=DS(0,0)=0\).

Define \(H(u,v):=v+S(u,v)\). For sufficiently small \(\epsilon>0\), we have $\text{Lip}(R)<\delta \ll 1$ and $\text{Lip}(H)<\mu \ll 1$. The local isochron is given by:
\[
\mathcal{I}_{loc}(x_{0})=\{(u,v)\in B_{\epsilon}(x_{0}) : |g^{m}(u,v)-x_{0}|<Ke^{-\lambda m}\}.
\]

We show that there exists a function \(\alpha:\mathbb{R}\to\mathbb{R}\) such that $\mathcal{I}_{loc}(x_{0})=\{(u,v): v=\alpha(u)\}=\text{graph}(\alpha)$, which is invariant under \(g\). Since $\text{Lip}(A^{-1})\text{Lip}(H)<1$ and $\text{Lip}(R)<\delta$, all hypotheses of the Graph Transform theorem are satisfied. Thus, there exists a unique function \(\alpha_{x_{0}}\) with $\alpha_{x_{0}}(0)=0$ and $\text{Lip}(\alpha_{x_{0}})\leq\delta$ such that $\text{graph}(\alpha_{x_{0}})=g(\text{graph}(\alpha_{x_{0}}))$.

The invariance under the flow follows from the group property and the continuity of \(\varphi\). For any \(x \in \mathcal{I}_{loc}(x_0)\), we have \(g^m(x) \to x_0\) exponentially fast. Then
\[
g^m(\varphi(t,x)) = \varphi(m+t,x) = \varphi(t,\varphi(m,x)) = \varphi(t,g^m(x)).
\]
By continuity of flow \(\varphi(t,g^m(x)) \to \varphi(t,x_0) = y_0\). Thus \(\varphi(t,\mathcal{I}_{loc}(x_0)) = \mathcal{I}_{loc}(y_0)\).

Distinct base points yield disjoint isochrons, which complete the foliation structure.
\end{proof}

\begin{remark}
The Graph Transform method, which is based on the geometric framework developed in \cite{gebrezabher2023thesis}, provides a clear geometric understanding of isochrons as stable manifolds of points on the limit cycle, forming a continuous invariant foliation of the basin of attraction. It intuitively refines a guess for the isochron by mapping it forward under the flow, and the exponential stability of the system guarantees that this process converges to the true invariant manifold—the isochron—which is represented as the graph of a function \(\alpha(u)\) in the local coordinates.
\end{remark}

\begin{figure}[!htb]
\centering

\includegraphics[width=0.7\textwidth]{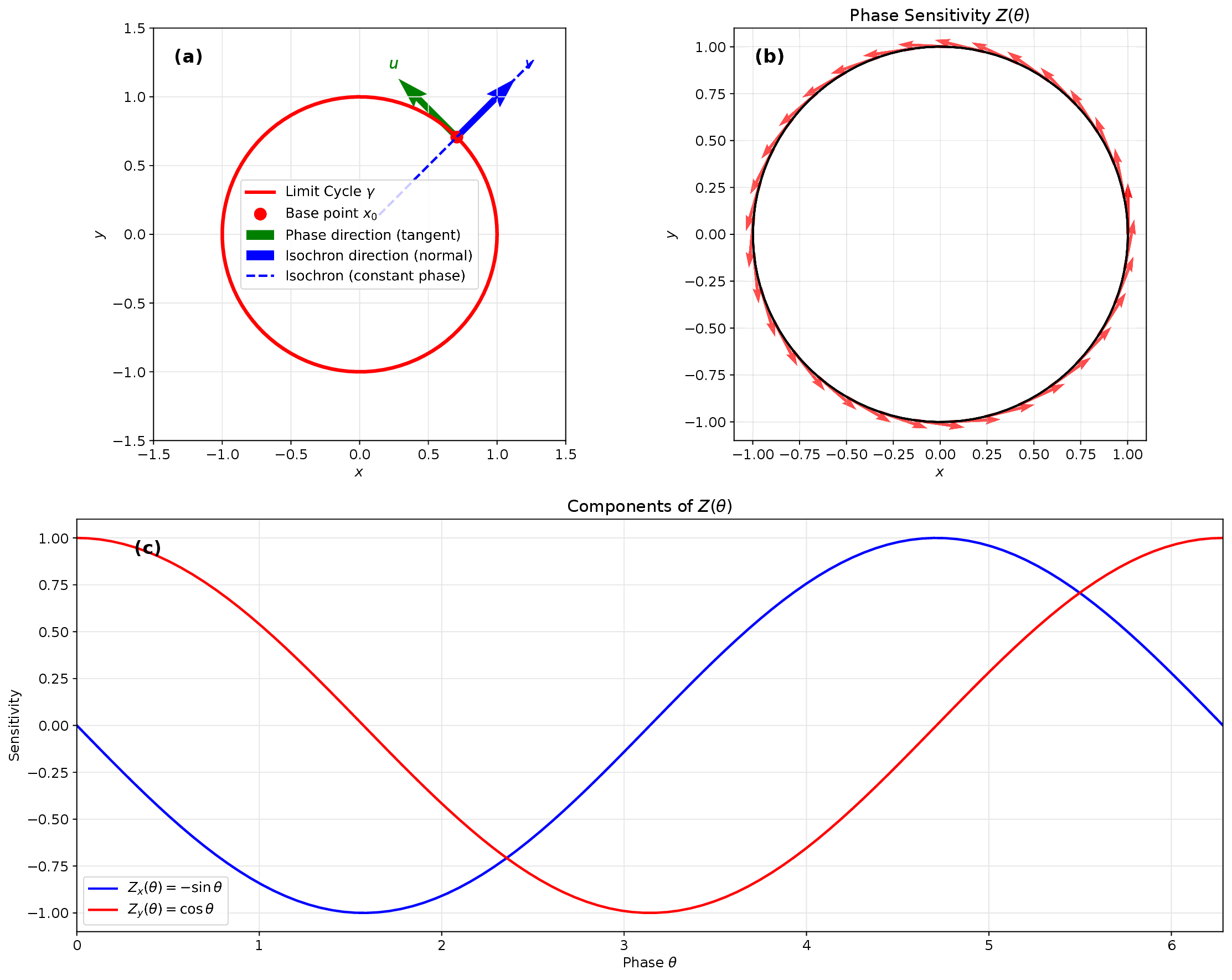}
\caption{\textbf{Local coordinate system and the phase sensitivity function ( $n=2$).} (a) A local coordinate frame 
\((u,v)\) is established near a base point $x_0$ on the limit cycle $\gamma$. The coordinate \(u\) corresponds to the phase direction (tangent to $\gamma$), while $v$ corresponds to the stable (amplitude) directions, transverse to $\gamma$. This coordinate system is essential for the constructive proof of isochron existence using the Graph Transform method (Theorem \ref{thm:isochrons}). (b) The phase sensitivity function (or iPRC) $Z(\theta)$ quantifies the infinitesimal phase shift caused by a perturbation. A perturbation $p$ can be decomposed into components tangential ($p_u$) and normal ($p_v$) to the limit cycle. The phase shift is given by the dot product \(Z(\theta)\cdot p\). The most effective perturbation for causing a phase shift is tangential to the isochron (and perpendicular to $\gamma$), as it does not displace the trajectory from its original asymptotic phase. (c) The components of the phase sensitivity function \(Z(\theta) = (Z_x(\theta), Z_y(\theta))\) are plotted against the phase $\theta$, showing how the oscillator's sensitivity to perturbations varies over its cycle.}
\label{fig:coordinate_transform}
\end{figure}

\begin{proposition}[Phase Dynamics \cite{winfree1967biological, kuramoto1984chemical}]
Let \(\varphi(t,x)\) be the solution of the system passing through \(x\) in \(\mathcal{B}(\gamma)\). Then the dynamics on \(\mathcal{B}(\gamma)\) of \(\gamma\) can be described by an asymptotic phase \(\theta(t):=\Theta(\varphi(t,x))\) in such a way that
\[
\frac{d\theta(t)}{dt}=1.
\]
\end{proposition}

\begin{proof}
From Theorem \ref{thm:isochrons}, the asymptotic phase \(\Theta(\varphi(t,x))\) increases uniformly in time \(t\) such that
\[
\frac{d}{dt}\Theta(\varphi(t,x))=1.
\]
By the chain rule, we obtain the phase equation
\[
\frac{d\theta(t)}{dt}=\nabla\Theta(\varphi(t,x))\cdot f(\varphi(t,x))=1
\]
for all points \(x\) in the basin of attraction \(\mathcal{B}(\gamma)\) of \(\gamma\).
\end{proof}

\subsection{Phase Sensitivity Function}

The vector function
\[
Z(\theta):=\nabla\Theta(x)|_{x=\gamma(\theta)}
\]
is called the \textbf{phase sensitivity function} or \textbf{infinitesimal phase response curve} \cite{winfree1967biological, ermentrout1996type, schultheiss2011phase}. It measures how sensitively the oscillator responds to external perturbations and plays a crucial role in phase reduction.

\subsection{Examples of Isochrons}

\begin{example}[Radial Isochrons]\label{ex:simple_isochrons}
Consider the dynamical system \cite{hoppensteadt1997weakly}
\[
\dot{x} = x - y - x(x^{2}+y^{2}), \quad \dot{y} = x + y - y(x^{2}+y^{2}).
\]
In polar coordinates with \(x=r\cos\phi\) and \(y=r\sin\phi\), we obtain
\[
\dot{r}=r(1-r^{2}); \quad \dot{\phi}=1.
\]
This system has an attracting limit cycle \(\gamma\) with radius \(r=1\), and basin of attraction \(\mathcal{B}(\gamma)=\mathbb{R}^{2}\setminus\{(0,0)\}\). The isochrons are lines
\[
\mathcal{I}_{\text{loc}}(x_0):=\{\phi=\theta\}.
\]
A few of these isochrons are shown in Figure \ref{fig:asymptotic_phase}(b).

\noindent\textbf{Gradient of the asymptotic phase.} The asymptotic phase is \(\Theta(r,\phi)=\phi=\theta\). The gradient is
\[
\nabla\Theta = \left(\frac{\partial\Theta}{\partial x}, \frac{\partial\Theta}{\partial y}\right).
\]
Using \(r=\sqrt{x^2+y^2}\) and \(\phi=\arctan(y/x)\), we have
\[
\frac{\partial\Theta}{\partial x}=-\frac{1}{r}\sin\theta, \quad \frac{\partial\Theta}{\partial y}=\frac{1}{r}\cos\theta.
\]
Along the limit cycle \(\gamma\), the gradient is
\[
Z(\theta):=\nabla\Theta=(-\sin\theta,\cos\theta).
\]
This function is plotted schematically in Fig. \ref{fig:coordinate_transform}(c).
\end{example}

\begin{example}[Spiral Isochrons]\label{ex:spiral_isochrons}
Consider the system \cite{guckenheimer1975isochrons}
\[
\dot{x}=x-(x+y)(x^{2}+y^{2}), \quad \dot{y}=y+(x-y)(x^{2}+y^{2}).
\]
In polar coordinates, the system becomes
\[
\dot{r}=r(1-r^{2}), \quad \dot{\phi}=r^{2}.
\]
There is an attracting periodic orbit \(\gamma\) with radius \(r=1\). The phase \(\phi\) along the orbit satisfies \(\dot{\phi}=1\). We use the ansatz
\[
\theta(t):=\Theta(r,\phi)=\phi-\eta(r),
\]
where \(\eta(r)\) is an unknown function to be determined.

Differentiating, we obtain
\[
\dot{\Theta}=\dot{\phi}-\frac{d\eta}{dr}\frac{dr}{dt}.
\]
For points \(x\in\mathcal{B}(\gamma)\), \(\dot{\theta}=1\). From the system equations,
\[
\frac{d\eta}{dr}=-\frac{1}{r}.
\]
So \(\eta(r)=-\log(r)+C\), where \(C\) is constant. Choosing \(C=0\), the isochrons are the set of points \((r,\phi)\) such that
\[
\Theta(r,\phi)=\phi+\log(r)=\theta.
\]
These isochrons are numerically verified and are shown in Figure \ref{fig:spiral_isochrons}.\\

\noindent\textbf{Gradient of the asymptotic phase.} The asymptotic phase is \(\Theta(r,\phi)=\phi+\log(r)\). Thus
\[
\frac{\partial\Theta}{\partial x} = \frac{\partial\Theta}{\partial r}\frac{\partial r}{\partial x}+\frac{\partial\Theta}{\partial\phi}\frac{\partial\phi}{\partial x} = \frac{1}{r}(\cos\phi-\sin\phi),
\]
\[
\frac{\partial\Theta}{\partial y} = \frac{\partial\Theta}{\partial r}\frac{\partial r}{\partial y}+\frac{\partial\Theta}{\partial\phi}\frac{\partial\phi}{\partial y} = \frac{1}{r}(\cos\phi+\sin\phi).
\]
Along the limit cycle \(\gamma\),
\[
Z(\theta):=\nabla\Theta=(\cos\theta-\sin\theta,\cos\theta+\sin\theta).
\] 
\end{example}

\begin{figure}[!htb]
    \centering
    \includegraphics[width=0.8\textwidth]{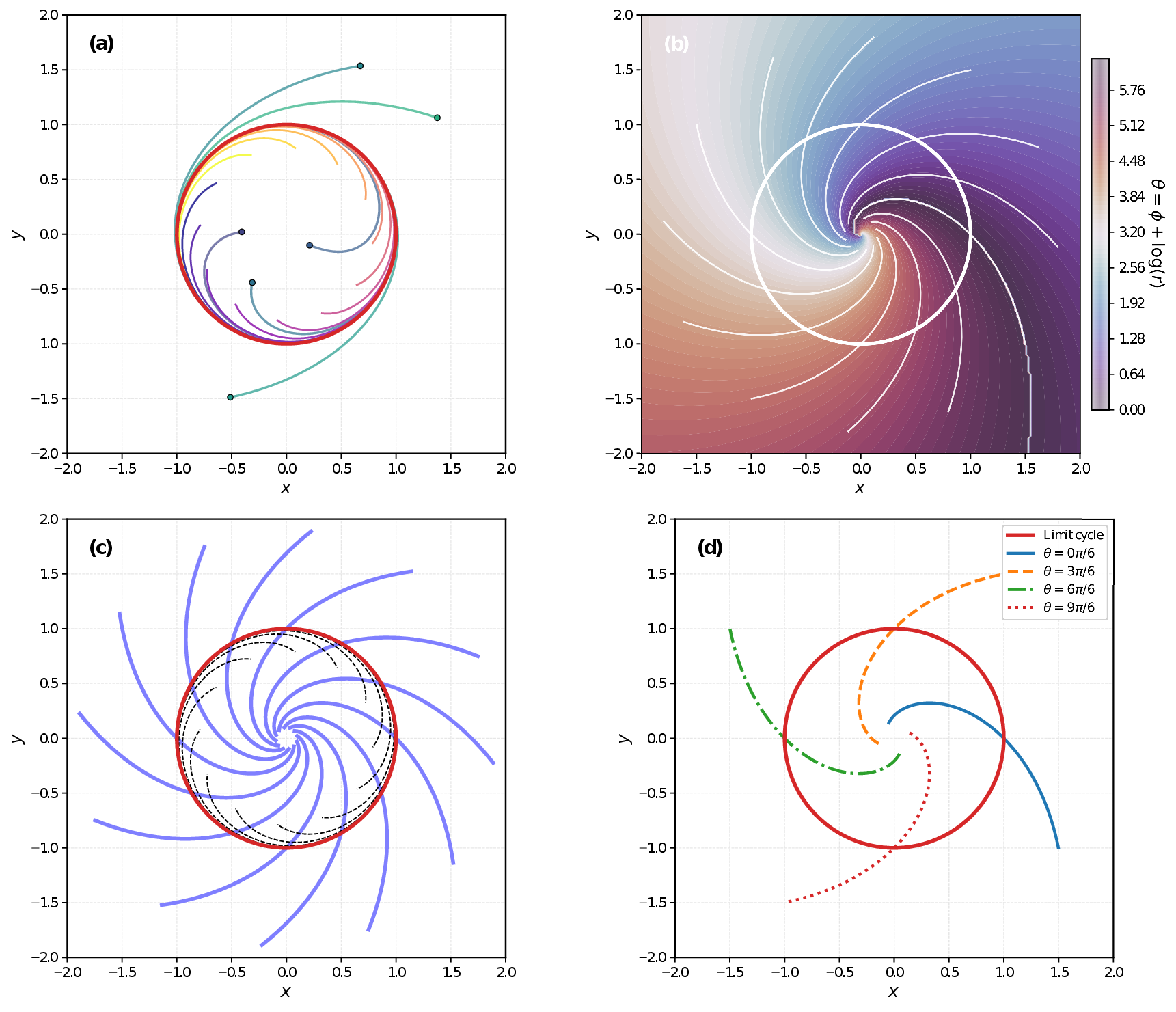}
    \caption{\textbf{Isochron analysis of the radial oscillator system from \Cref{ex:spiral_isochrons}.} The system dynamics are governed by $\dot{r} = r(1-r^2)$, $\dot{\phi} = r^2$ in polar coordinates, with a stable limit cycle at $r=1$. (a) Numerical isochrons (colored curves) computed via backward integration \citep{guckenheimer1975isochrons, Winfree2001}, with sample trajectories showing evolution toward the limit cycle. Circles indicate initial conditions. (b) Phase field $\theta(x,y) = \phi + \log(r)$ with analytical isochrons (white curves) as level sets. (c) Comparison showing perfect overlap between analytical isochrons (blue, $\phi + \log(r) = \text{constant}$) and numerical results (black dashed). (d) Selected isochrons with different asymptotic phases $\theta$, demonstrating the logarithmic spiral structure $\phi + \log(r) = \theta$.}
    \label{fig:spiral_isochrons}
\end{figure}

\section{First-Order Phase Reduction for Weakly Perturbed Systems}\label{sec:phase-reduction}

\subsection{Persistence of Periodic Orbits under Perturbation}

We consider the system being perturbed as
\begin{equation}\label{eq:perturbed_system}
\frac{dx}{dt}=f(x)+\epsilon p(x,t), \quad x\in\mathbb{R}^{n},
\end{equation}
where \(p(x,t+T)=p(x,t)\) for all \(t\), and \(\epsilon\) is a small parameter, where $T$ is the period of the perturbation.

When \(p\) is only a function of \(x\), i.e., \(p=p(x)\), the periodic orbit \(\gamma\) of \(f\) persists by the implicit function theorem on the Poincaré map \cite{guckenheimer1983nonlinear}.

\begin{proposition}[Persistence of Periodic Orbits \cite{guckenheimer1983nonlinear}]
Let \(\Sigma\) be a local transversal section of \(\gamma\) and \(P:\Sigma\to\Sigma\) the corresponding Poincaré map. By construction, \(P(x_0)=x_0\) for all \(x_0\in\gamma\) and the Jacobian \(D_x P(x_0)\) has no eigenvalue 1. For small \(\epsilon\), a local Poincaré map \(P_\epsilon:\Sigma\to\Sigma\) is well defined. Its fixed points \(x_\epsilon\) correspond to periodic orbits \(\gamma_\epsilon\).
\end{proposition}

\begin{proof}
By the implicit function theorem, there exists a local arc \(\epsilon\mapsto x(\epsilon)\) with \(x(0)=x_0\) such that \(P_\epsilon(x(\epsilon))=x(\epsilon)\), showing the persistence of the periodic orbit \cite{guckenheimer1983nonlinear}.
\end{proof}

\subsection{Phase Reduction Theorem}

Under a weak perturbation \(\epsilon p(x,t)\), any solution of the perturbed system \eqref{eq:perturbed_system} that starts in the neighborhood of an exponentially stable limit cycle \(\gamma\) stays in its neighborhood. Let us introduce a phase variable \(\vartheta=\Omega t\) to transform the system \eqref{eq:perturbed_system} into an autonomous system on the extended phase space \(\mathbb{R}^n\times\mathbb{S}^1\) \cite{kuramoto1984chemical} as
\begin{equation}\label{eq:extended_system}
\frac{dx}{dt}=f(x)+\epsilon p(x,\vartheta), \quad \frac{d\vartheta}{dt}=\Omega
\end{equation}
with \(p(x,\vartheta+2\pi)=p(x,\vartheta)\).

A transversal cross-section can be defined as
\[
\Sigma=\{(x,\vartheta)\in\mathbb{R}^{n}\times\mathbb{S}^{1} : \vartheta=\vartheta_{0}\}.
\]

\noindent\textbf{Underlying assumptions}:
\begin{enumerate}
\item[(A1)] \(f\) and \(p\) are continuously differentiable
\item[(A2)] the unperturbed system has an exponentially stable limit cycle \(\gamma\subset\mathbb{R}^{n}\) with period 1
\end{enumerate}

\begin{theorem}[Perturbed System Invariance]\label{thm:perturbed_invariance}
Consider the extended system \eqref{eq:extended_system} satisfying assumptions (A1) and (A2). Then \(\exists\epsilon_0>0\) such that \(\forall\epsilon<\epsilon_0\), there is a neighborhood \(W\subset\mathbb{R}^{n}\) of \(\gamma\) that is positively invariant:
\[
g_{\epsilon}(W)\subset W,
\]
where \(g_{\epsilon}\) is the time-one map of the perturbed system.
\end{theorem}

\begin{proof}
Assume \(\epsilon \ll 1\). The flow of the perturbed system is \(\varphi_\epsilon(t,x,\vartheta)=(\varphi_\epsilon(\vartheta,x,t,\epsilon),\vartheta(t))\). Define the time-one map as
\[
g_\epsilon:=\varphi_\epsilon(1,\cdot,\cdot):\mathbb{R}^{n}\times\mathbb{S}^{1}\longrightarrow\mathbb{R}^{n}\times\mathbb{S}^{1}, \quad (x,\vartheta)\longmapsto(g_\epsilon(x,\vartheta),\bar{\vartheta}).
\]
For \(\epsilon=0\), \(g_0(x,\vartheta)=g(x,\vartheta)\), where \(g:\mathbb{R}^{n}\to\mathbb{R}^{n}\) is the time-one map for the unperturbed system. By construction, \(g\) is a contraction with Lipschitz constant \(k\in(0,1)\) and unique fixed point \(z\), and \(g(\gamma)=\gamma\).

Let us introduce the composition $$g_{\epsilon}^{m}(\cdot,\bm{\vartheta})=g_\epsilon(\cdot,\vartheta_m)\circ g_\epsilon(\cdot,\vartheta_{m-1})\circ \dots g_\epsilon(\cdot,\vartheta_1)$$ of transformations, and assume that the transformations satisfy
\begin{align}
    \sup_{x\in \mathbb{R}^n} \|g(x) - g_{\epsilon}(x,\theta)\| \leq \epsilon.
\end{align}{}
Consider a fixed $\vartheta$ and let $G=g_{\epsilon}$, then note that
\[
\|G(x) - G(y)\| \le \|G(x) - g(x)\| + \|g(y) - G(y)\| + \|g(x) - g(y)\| \le  2 \varepsilon + k \| x - y\|
\]
and,
\[
\|G(x) - g(y)\| \le \|G(x) - G(y)\| + \| g(y) -  G(y) \| \le  3 \varepsilon + k \| x - y  \|.
\]
Consider a ball  $x \in W(z,\delta)$.  We claim that if $ \delta = \frac{3 \varepsilon}{1 - k }$ then 
$g^n_{\epsilon}(x,\vartheta) \in W(z,\delta)$. Indeed, consider the action of a generic element $G$
\[
\| G(x) - z \|  = \| G(x) -  g(z) \| \le 3 \varepsilon + k \| x - z \|  \le \delta
\]
but since $x \in W(z,\delta)$ we have the following bound $\delta \ge 3 \varepsilon/(1 - k )$, and by induction the claim follows. This shows that $g_\epsilon(x,\theta)$ is a contraction on $W$.
\end{proof}
\Cref{thm:perturbed_invariance} implies that the perturbed orbit, denoted by $\Tilde{\gamma}$ projected to $\mathbb{R}^n$ stays in a neighborhood of size $O(\epsilon)$ of the unperturbed orbit, $\gamma$. This observation leads us to obtain an approximated orbit equation under perturbation, demonstrated in \cref{fig:perturbedSoln}.

\begin{figure}[!ht]
    \centering
    \includegraphics[width=0.8\textwidth]{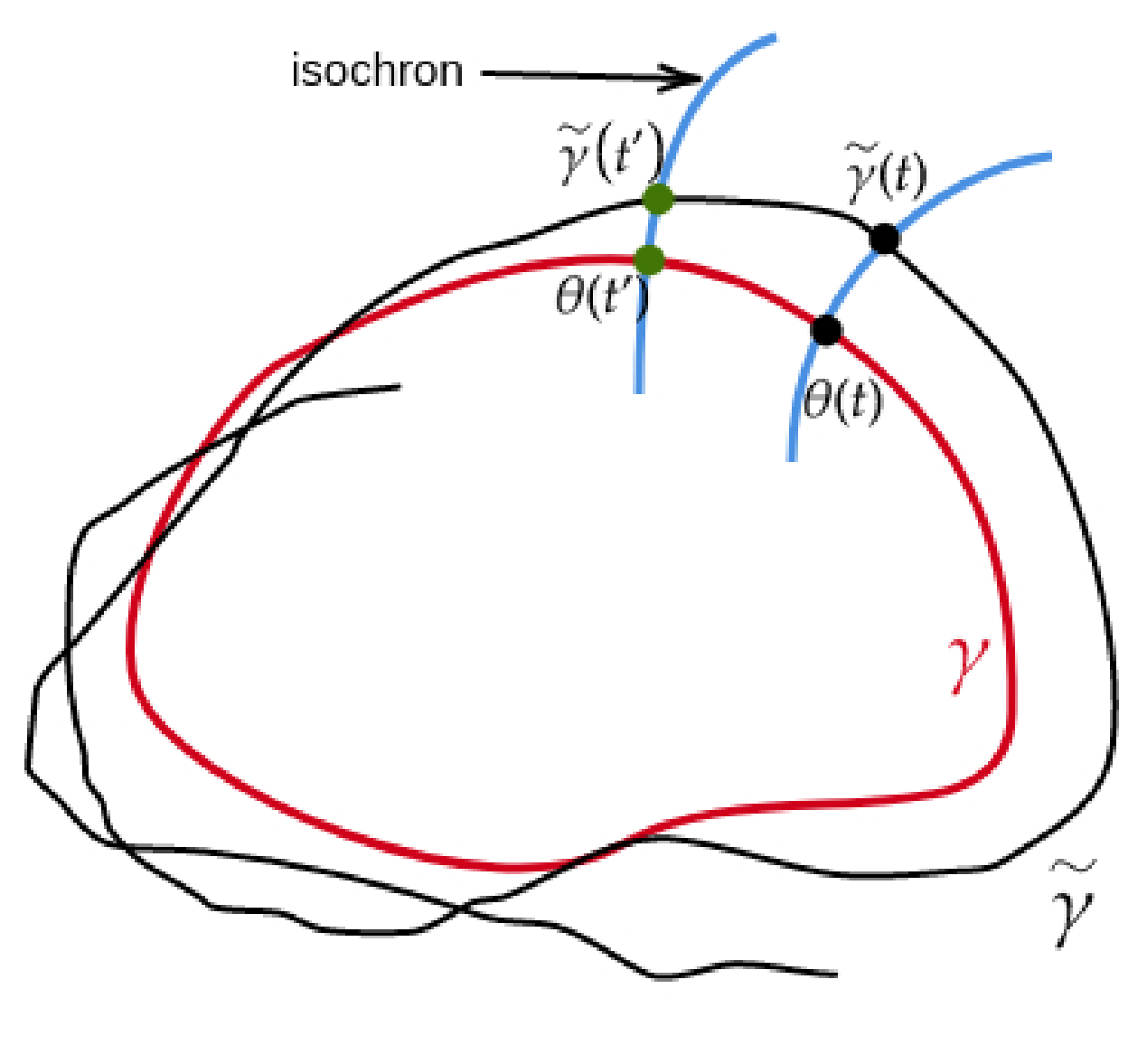}
    \caption{Schematic illustrating the core idea of phase reduction. A trajectory \(\tilde{\gamma}(t)\) of the perturbed system (black) remains in an \(\mathcal{O}(\epsilon)\) neighborhood of the unperturbed limit cycle \(\gamma\) (red). The phase reduction approximates its dynamics by the phase of the nearest point on \(\gamma\), effectively projecting the state onto the limit cycle.}
    \label{fig:perturbedSoln}
\end{figure}

\begin{theorem}[First-Order Phase Reduction \citep{kuramoto1984chemical, winfree1967biological}]\label{thm:phase_reduction}
Consider the system \eqref{eq:perturbed_system} with \(\tilde{\gamma},\gamma\) as defined above, and let \(\gamma(\theta)\) be a point on the limit cycle \(\gamma\) with phase \(\theta\in\mathbb{S}^1\). Assume \(\gamma(\theta)\), \(\theta\in\mathbb{S}^1\), and \(\epsilon\ll 1\). Then the dynamics is reduced to a phase model
\begin{equation}\label{eq:reduced_eqn}
\frac{d\theta}{dt}=1+\epsilon\Gamma(\theta,t)+\epsilon^2 R(\theta,t,\epsilon),
\end{equation}
where
\[
\Gamma(\theta,t):=\nabla\Theta|_{\gamma(\theta)}\cdot p(\gamma(\theta),t)
\]
and \(R\) is the remainder term.
\end{theorem}

\begin{proof}
By construction, \(\theta(t)=\Theta(\tilde{\gamma}(t))\) is a phase in \(\gamma\). Using the chain rule and the perturbed system \eqref{eq:perturbed_system}, we obtain
\begin{equation}\label{eq:reduced_nonauto_system}
\frac{d\theta(t)}{dt}=\nabla\Theta(\tilde{\gamma}(t))\cdot[f(\tilde{\gamma}(t))+\epsilon p(\tilde{\gamma}(t),t)]=1+\epsilon\nabla\Theta(\tilde{\gamma}(t))\cdot p(\tilde{\gamma}(t),t),
\end{equation}
where we used \(\nabla\Theta(x)\cdot f(x)=1\) in the basin of attraction of \(\gamma\). Notice that there is no approximation in equation~\eqref{eq:reduced_nonauto_system} as it depends explicitly on $x, t$, and $\theta$. We want to reduce it to an equation of $\theta$ only.

Since the isochrons $\mathcal{I}_{loc}(x)$ smoothly depend on the base point $x$, then the gradient $\nabla\Theta(x)$ is differentiable for each $x$ in a neighborhood of $\gamma$. We know that $\Tilde{\gamma}(t)$ is assumed to be $\epsilon$-close to the limit cycle $\gamma$, that is, $|\Tilde{\gamma}(t)-\gamma(t)| = \mathcal{O}(\epsilon)$. By the Taylor expansion of $\nabla\Theta$ about $\Tilde{\gamma}(t)-\gamma(t)$, we have
\[
\nabla\Theta(\tilde{\gamma}(t)) = \nabla\Theta(\gamma(t)+\Tilde{\gamma}(t)-\gamma(t))=\nabla\Theta(\gamma(t))+ \mathcal{O}(\epsilon). 
\]
From assumption (A1) we know that $p$ is differentiable, and hence by the Taylor expansion of $p$ about $\Tilde{\gamma}(t)-\gamma(t)$, we have
\[
p(\tilde{\gamma}(t),t) = p(\gamma(t)+\Tilde{\gamma}(t)-\gamma(t),t)=p(\gamma(t),t)+\mathcal{O}(\epsilon).
\]
Hence, the first-order approximation phase dynamics of system \eqref{eq:reduced_nonauto_system} along the orbit $\gamma$ can be written as
\begin{align}\label{eq:perturbed_reducedEq}
    \frac{d\theta(t)}{d t} = 1 + \epsilon\nabla\Theta(\gamma(t))\cdot p(\gamma(t),t) + \mathcal{O}(\epsilon^2)
\end{align}{}
which is the reduced phase equation of the perturbed system \eqref{eq:perturbed_system}. 
\end{proof}

\begin{remark}
The systematic derivation of first-order phase reduction presented here extends the treatment in \cite{gebrezabher2023thesis}, which also addresses practical considerations for experimental applications and network reconstruction from data.
\end{remark}

\begin{figure}[!ht]
\centering
\includegraphics[width=0.8\textwidth]{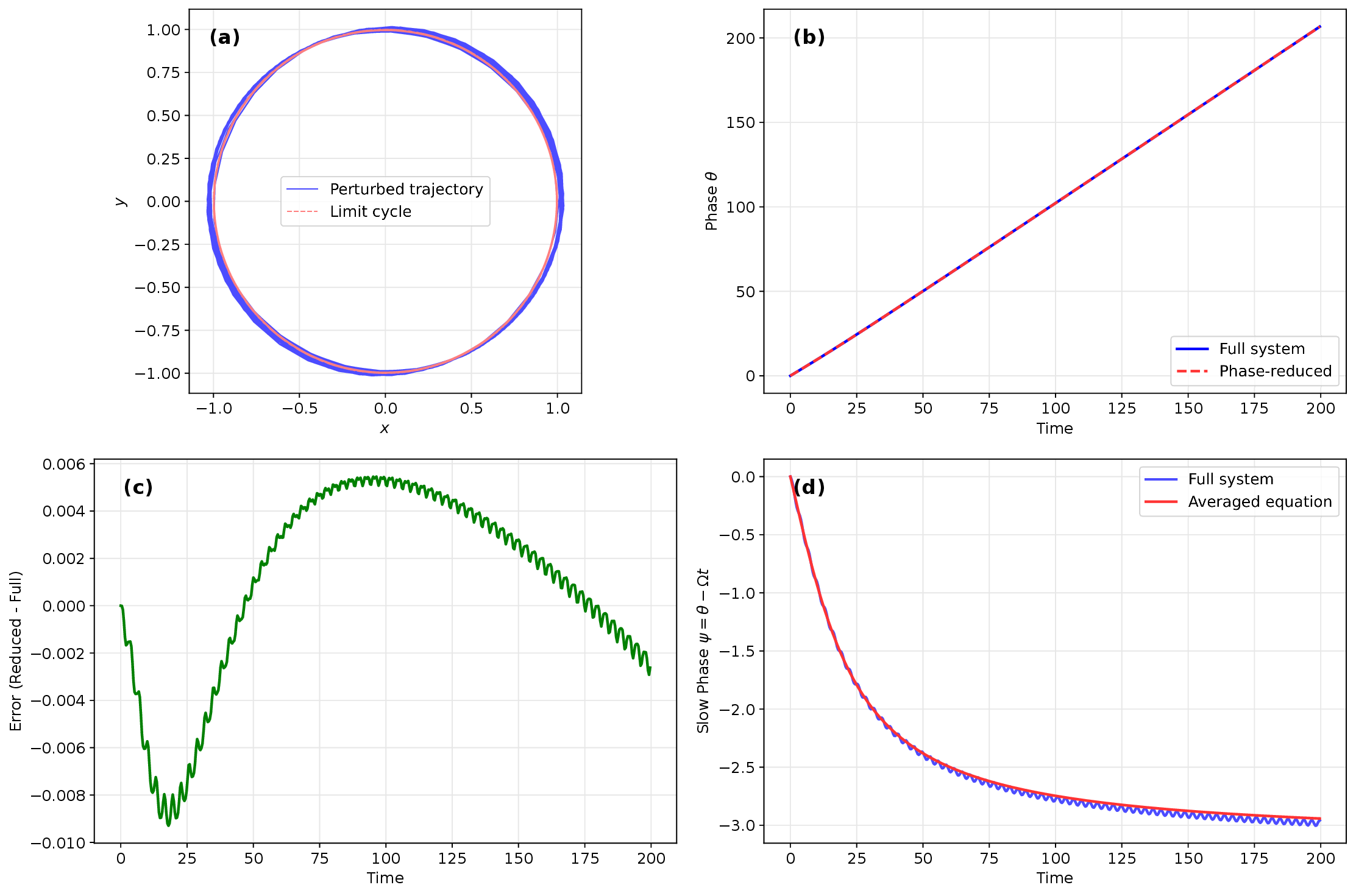}
\caption{\textbf{Numerical validation of the first-order phase reduction for a weakly perturbed oscillator.}
(a) Trajectory of the perturbed system (blue) remains within an $\mathcal{O}(\epsilon)$ neighborhood of the unperturbed limit cycle $\gamma$ (red). The phase reduction approximates the dynamics by projecting the state onto the limit cycle along the local isochron (dashed line). (b) Comparison of the phase evolution 
$\theta(t)$ between the full system (solid line) and the first-order reduced model (dashed line). The close agreement validates the reduction. (c) The absolute error \(|\theta_{full}(t)-\theta_{reduced}(t)|\) remains small, of order $\mathcal{O}(\epsilon)$, over the observed time scale, confirming the theoretical prediction. (d) Evolution of the slow phase variable 
$\psi(t)=\theta(t)-\Omega t$. Its slow dynamics, compared to the rapid oscillation of $\theta(t)$, demonstrate the validity of the averaging approximation, which separates the fast and slow time scales.}
\label{fig:perturbed_trajectory}
\end{figure}

\subsection{Effective Coupling on Slow Scales}

In the reduced phase equation, considering \(\vartheta:=\Omega t\) as the phase of the external influence with \(\Omega\approx 1\), we define the slow phase
\[
\psi=\theta-\Omega t.
\]
Differentiating with respect to \(t\), we obtain
\begin{align}\label{eq:reduced_slow_phase}
\frac{d\psi(t)}{dt}=(1-\Omega)+\epsilon \Gamma(\psi,t) + \mathcal{O}(\epsilon^2),
\end{align}
where \(\Gamma(\psi,t)=Z(\psi+\Omega t)\cdot p(\gamma(\psi+\Omega t),t)\).

Assuming \(\Omega\approx 1\), we have \(1-\Omega=\epsilon\Delta\) with \(\Delta=\mathcal{O}(1)\). Thus \(\dot{\psi}=\mathcal{O}(\epsilon)\), \(\dot{\vartheta}=O(1)\), and \(\psi(t)\) is much slower than \(p(t)\). We can approximate \(\Gamma(\psi,t)\) by its average \cite{sanders1985averaging}.

By the averaging \cref{thm:periodic-averaging}, we approximate the right-hand side by integrating over one period of the fast external forcing \(p\), assuming \(\psi(t)\) does not vary within its period \(T\) and obtain an averaged phase equation of the form
\[
\frac{d\psi(t)}{dt}=\epsilon\Delta+\epsilon\bar{\Gamma}(\psi),
\]
where the coupling function is
\[
\bar{\Gamma}(\psi)=\frac{1}{T}\int_{0}^{T}\Gamma(\psi,t)dt=\frac{1}{T}\int_{0}^{T}Z(\psi+\Omega t)\cdot u(\psi+\Omega t,t)dt.
\]
where we have used the definition \(u(\psi+\Omega t,t) = p(\gamma(\psi+\Omega t),t)\). This result is summarized in the following theorem.

\begin{theorem}[Averaged Phase Reduction \cite{sanders1985averaging, kuramoto1984chemical}]\label{thm:averaged_reduction}
Consider the reduced phase equation~\eqref{eq:reduced_slow_phase}, with \(\epsilon\ll 1\), and \(\psi=\theta-\Omega t\) is a slowly varying phase. On a time scale of \(t\sim 1/\epsilon\), the equation can be approximated by
\[
\frac{d\psi}{dt}=\epsilon\Delta+\epsilon\bar{\Gamma}(\psi),
\]
where
\[
\bar{\Gamma}(\psi)=\frac{1}{T}\int_{0}^{T}Z(\psi+\Omega t)\cdot u(\psi+\Omega t,t)dt.
\]
Moreover, the solutions are related by \(|\psi(t)-\phi(t)|=\mathcal{O}(\epsilon)\), where \(\phi(t)\) is the solution of the averaged equation.
\end{theorem}

\subsection{Phase Locking}

\begin{proposition}[Phase Locking Condition for Sinusoidal Coupling \cite{winfree1967biological, kuramoto1984chemical}]
Consider an oscillator described by the phase equation in a rotating frame with frequency \(\Omega\):
\[
\frac{d\psi}{dt} = (1 - \Omega) - \epsilon F(\psi),
\]
where \(\psi = \theta - \Omega t\) is the phase difference, and \(\epsilon > 0\) is the coupling strength. Assume the coupling function is \(F(\psi) = \sin(\psi)\).

The system is in a \textbf{phase-locked} state if the phase difference \(\psi\) is constant, i.e., \(\frac{d\psi}{dt}=0\). This state corresponds to a fixed point \(\psi^*\) satisfying
\[
\sin(\psi^*) = \frac{1 - \Omega}{\epsilon}.
\]
A stable fixed point exists when the phase locking condition \(|1 - \Omega| < \epsilon\) is satisfied. The two fixed points in one period \([0, 2\pi)\) are:
\[
\psi_1^* = \arcsin\left(\frac{1 - \Omega}{\epsilon}\right), \quad \psi_2^* = \pi - \arcsin\left(\frac{1 - \Omega}{\epsilon}\right).
\]
The stable fixed point is \(\psi_1^*\).
\end{proposition}

\begin{proof}
We find the fixed points by setting the right-hand side of the phase equation to zero:
\[
\frac{d\psi}{dt} = (1 - \Omega) - \epsilon \sin(\psi) = 0.
\]
Solving for \(\sin(\psi)\) gives the fixed point condition:
\[
\sin(\psi^*) = \frac{1 - \Omega}{\epsilon}. \tag{1}
\]

For a solution \(\psi^*\) to exist, the right-hand side must lie within the range of the sine function, and hence we must have
\[
\left| \frac{1 - \Omega}{\epsilon} \right| \leq 1.
\]
Multiplying both sides by \(\epsilon\) yields the necessary condition for phase locking:
\[
|1 - \Omega| \leq \epsilon.
\]
Assuming the strict inequality \(|1 - \Omega| < \epsilon\) holds, there are two solutions in one period \([0, 2\pi)\), as stated in the proposition: \(\psi_1^*\) in the \(4^{\text{th}}\) quadrant (or \(1^{\text{st}}\) if \(1-\Omega<0\)) and \(\psi_2^*\) in the \(2^{\text{nd}}\) quadrant (or \(3^{\text{rd}}\) if \(1-\Omega<0\)).

To determine stability, we examine the derivative of the phase dynamics with respect to the phase $\psi$ as
\[
\frac{d}{d\psi}\left(\frac{d\psi}{dt}\right) = -\epsilon \cos(\psi).
\]
A fixed point is stable if this derivative is negative.
\begin{itemize}
	\item At \(\psi_1^*\), \(\cos(\psi_1^*) > 0\), so the derivative is \(-\epsilon \cos(\psi_1^*) < 0\). This point is stable.
	\item At \(\psi_2^*\), \(\cos(\psi_2^*) < 0\), so the derivative is \(-\epsilon \cos(\psi_2^*) > 0\). This point is unstable.
\end{itemize}

Therefore, the stable phase-locked state is given by the fixed point where \(\cos(\psi^*) > 0\), which is \(\psi_1^* = \arcsin\left(\frac{1 - \Omega}{\epsilon}\right)\).
\end{proof}

\subsection{Numerical Illustration of Phase Reduction Validity}

Consider the dynamical system from Example \ref{ex:simple_isochrons}:
\[
\dot{\bm{x}}=f(\bm{x})=\begin{pmatrix}x-y-x(x^{2}+y^{2})\\ x+y-y(x^{2}+y^{2})\end{pmatrix},
\]
where \(\bm{x}=(x,y)^T\). The phase sensitivity function is \(Z(\theta)=(-\sin\theta,\cos\theta)\).

A small periodic perturbation
\[
\epsilon p(\bm{x},t)=\epsilon\sin(\Omega t), \quad \epsilon\ll 1
\]
is applied in the \(x\) direction. Then, the perturbed system is given by
\[
\dot{\bm{x}}=f(\bm{x})+\epsilon\begin{pmatrix}\sin(\Omega t)\\ 0\end{pmatrix}.
\]
Assume \(\Omega\approx 1\).

The phase dynamics of the perturbed system is given by
\[
\frac{d\theta(t)}{dt}=1+\epsilon Z(\theta)\cdot p(t)+\mathcal{O}(\epsilon^2)=1-\epsilon\sin\theta\sin(\Omega t)+\mathcal{O}(\epsilon^2).
\]
Introducing \(\psi=\theta-\Omega t\):
\[
\frac{d\psi(t)}{dt}=1-\Omega-\epsilon\sin(\psi+\Omega t)\sin(\Omega t)+O(\epsilon^2).
\]
Averaging over period \(T=\frac{2\pi}{\Omega}\), we get
\[
\frac{d\psi(t)}{dt}=1-\Omega-\epsilon\bar{\Gamma}(\psi), \quad \bar{\Gamma}(\psi)=\frac{1}{2}\cos\psi.
\]
The validity of this first-order averaged phase equation is confirmed by numerical simulation. As shown in \Cref{fig:perturbed_trajectory}, the phase dynamics of the full system are accurately captured by the reduced model, with the error scaling as \(\mathcal{O}(\epsilon)\). Furthermore, the evolution of the slow phase variable \(\psi(t)\) demonstrates the time-scale separation that justifies the averaging procedure.

\begin{figure}[!ht]
\centering
\includegraphics[width=0.95\textwidth]{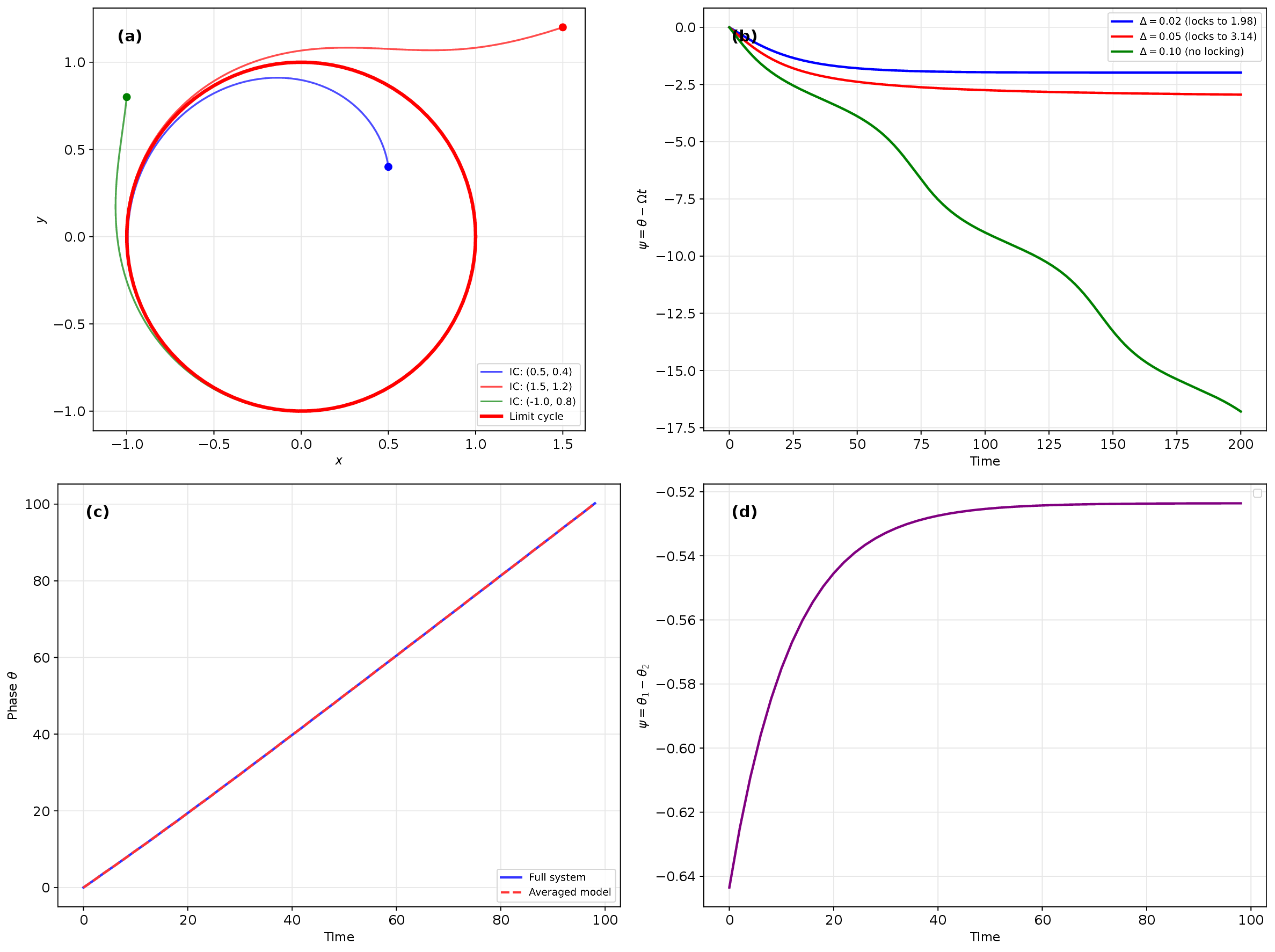}
\caption{\textbf{Phase locking and the accuracy of the averaged phase model.} (a) Stable limit cycle (red) in the $(x,y)$-plane with trajectories (colored lines) converging from different initial conditions. (b) Phase locking behavior for different frequency detunings $\Delta = \Omega-1$ under weak periodic forcing. When $|\Delta|< \epsilon/2$, stable phase-locked states emerge (blue and red curves), where the oscillator frequency entrained to the forcing frequency. For larger detunings (green curve), phase locking does not occur. (c) Comparison between the exact phase dynamics of the full, perturbed system and the dynamics predicted by the averaged phase model (Theorem \ref{thm:averaged_reduction}). The close match demonstrates the model's effectiveness. (d) Evolution of the phase difference \(\psi=\theta_1 - \theta_2\) for two weakly coupled oscillators (as in Section 5.1). The convergence to a stable, constant phase difference $\psi^*$ illustrates the phenomenon of phase locking in a network context.}
\label{fig:phase_locking_accuracy}
\end{figure}

\section{Applications to Coupled Oscillator Networks}
\label{sec:applications}

The formalism developed for a single weakly forced oscillator extends naturally to networks of coupled oscillators. Although we initially consider periodic coupling for clarity, many realistic network interactions exhibit quasi-periodic or almost-periodic behavior due to heterogeneous oscillator frequencies or complex external inputs. The almost-periodic averaging theory developed in Section 2.2 provides the mathematical foundation for these more general cases.

\subsection{Two Weakly Coupled Oscillators}

\begin{theorem}[Phase Reduction for Two Weakly Coupled Oscillators \cite{kuramoto1984chemical, ashwin2016phase, winfree1967biological}]
    Consider two weakly coupled oscillators
\[
\frac{dx_{1}}{dt}=f_{1}(x_{1})+\epsilon g_{1}(x_{1},x_{2}), \quad \frac{dx_{2}}{dt}=f_{2}(x_{2})+\epsilon g_{2}(x_{1},x_{2}),
\]
where \(x_{i}\in\mathbb{R}^{n}\), and \(\epsilon \ll 1\). Assume each isolated system \(\dot{x}_i = f_i(x_i)\) has an exponentially stable limit cycle \(\gamma_i\) with natural frequency \(\Omega_i\). Let \(\Theta_i\) be the asymptotic phase function and \(Z_i(\theta_i) = \nabla\Theta_i|_{\gamma_i(\theta_i)}\) be the phase sensitivity function for each oscillator. Furthermore, assume the natural frequencies are close, i.e., \(|\Omega_1 - \Omega_2| = \mathcal{O}(\epsilon)\).

Then, to first order in \(\epsilon\), the dynamics of the coupled system are described by the phase model
\[
\frac{d\theta_{1}}{dt}=\Omega_{1}+\epsilon \bar{q}_{1}(\theta_{2}-\theta_{1}) + \mathcal{O}(\epsilon^{2}), \quad 
\frac{d\theta_{2}}{dt}=\Omega_{2}+\epsilon \bar{q}_{2}(\theta_{1}-\theta_{2}) + \mathcal{O}(\epsilon^{2}),
\]
where the coupling functions \citep{Stankovski2017coupling} \(\bar{q}_{i}\) are given by
\[
\bar{q}_{i}(\phi)=\frac{1}{T}\int_{0}^{T}Z_{i}(s)\cdot g_{i}(\gamma_{i}(s),\gamma_{j}(s+\phi))\,ds, \quad i=1,2.
\]
Here, \(T\) is a suitable averaging period.
\end{theorem}

\begin{proof}
The derivation follows from applying the chain rule to \(\dot{\theta}_i = \nabla\Theta_i(x_i) \cdot \dot{x}_i\), evaluating the coupling terms on the unperturbed limit cycles \(\gamma_i(\theta_i)\) and \(\gamma_j(\theta_j)\), and applying the averaging \cref{thm:periodic-averaging} to the resulting equations \cite{sanders1985averaging, kuramoto1984chemical}.. The assumption of close natural frequencies ensures that the phase difference \(\theta_j - \theta_i\) evolves slowly, making the averaging valid over the fast period \(T\). 
\end{proof}
This phase reduction reveals the fundamental mechanism of synchronization. The dynamics of the phase difference \(\phi = \theta_2 - \theta_1\) evolve on a slow timescale, described to first order by the equation:
\[
\frac{d\phi}{dt} = (\Omega_2 - \Omega_1) + \epsilon \Gamma(\phi) + \mathcal{O}(\epsilon^2),
\]
where \(\Gamma(\phi) = \bar{q}_2(-\phi) - \bar{q}_1(\phi)\) is the effective coupling or phase interaction function. Synchronization, in the form of a stable fixed point \(\phi^*\) where \(d\phi/dt = 0\), occurs when the frequency detuning \(\Omega_2 - \Omega_1\) is balanced by the average coupling force \(\epsilon \Gamma(\phi^*)\). The stability of this phase-locked state is determined by the slope of \(\Gamma(\phi)\) at \(\phi^*\). This powerful result tells us that the complex interaction of the two oscillators \(g_i(x_1, x_2)\) is distilled into the simple, periodic function \(\Gamma(\phi)\), which dictates all possible synchronization regimes. These predicted dynamics are illustrated in \cref{fig:coupled_oscillators}.

\subsection{Networks of Weakly Coupled Oscillators}

The phase reduction approach for oscillator networks developed here provides the theoretical foundation for the network reconstruction methods presented in \cite{gebrezabher2023thesis}. In that work, techniques for inferring phase dynamics and coupling functions from experimental time series data are developed and applied to both synthetic and real-world oscillator networks.

\begin{theorem}[Phase Reduction for a Network of Weakly Coupled Oscillators \cite{ashwin2016phase, nakao2016phase, kuramoto1984chemical}]
    Consider a network of \(N\) oscillators governed by
\[
\frac{d}{dt}x_{i}(t)=f_{i}(x_{i})+\epsilon\sum_{j=1}^{N}A_{ij}\,h(x_{i},x_{j}), \quad i=1,\ldots,N,
\]
where \(x_{i}(t)\in U_{i}\subset\mathbb{R}^{n}\), and \(\epsilon \ll 1\). Assume each isolated system \(\dot{x}_i = f_i(x_i)\) has an exponentially stable limit cycle \(\gamma_i\) with natural frequency \(\Omega_i\). Let \(\Theta_i\) be the asymptotic phase function and \(Z_i(\theta_i) = \nabla\Theta_i|_{\gamma_i(\theta_i)}\) be the phase sensitivity function for each oscillator. Furthermore, assume the natural frequencies are close, i.e., \(|\Omega_i - \Omega_j| = \mathcal{O}(\epsilon)\).
Then, to first order in \(\epsilon\), the dynamics of the network are described by the phase model
\[
\frac{d\theta_{i}}{dt}=\Omega_{i}+\epsilon\sum_{j=1}^{N}A_{ij}\,\bar{q}_{ij}(\theta_{j}-\theta_{i}) + \mathcal{O}(\epsilon^{2}), \quad i=1,\ldots,N,
\]
where the coupling functions \(\bar{q}_{ij}\) are given by
\[
\bar{q}_{ij}(\phi)=\frac{1}{T}\int_{0}^{T}Z_{i}(s)\cdot h(\gamma_{i}(s),\gamma_{j}(s+\phi))\,ds.
\]
Here, \(T\) is a suitable averaging period (e.g., the period of the mean-field oscillation).
\end{theorem}

\begin{proof}
The derivation follows the same pattern as the two-oscillator case, with the additional summation over network connections and application of the averaging theorem to each coupling term \cite{ashwin2016phase, nakao2016phase}. The network phase reduction formalism extends the classical Kuramoto model \cite{kuramoto1984chemical} to general coupled oscillator systems.
\end{proof}

\subsection{Networks with Almost-Periodic Coupling}

In many realistic scenarios, network interactions are not strictly periodic but exhibit almost-periodic behavior due to heterogeneous oscillator frequencies, time-varying coupling strengths, or external multi-frequency forcing. The almost-periodic averaging theory developed in Section 2.2 provides the mathematical framework for these cases.

\begin{theorem}[Phase Reduction for Networks with Almost-Periodic Coupling]\label{thm:almost_periodic_network}
    Consider a network of oscillators with almost-periodic coupling given by
\begin{equation}\label{eq:almost_periodic_network}
\frac{d}{dt}x_{i}(t)=f_{i}(x_{i})+\epsilon\sum_{j=1}^{N}A_{ij}(t)h(x_{i},x_{j},t), \quad i=1,\ldots,N,
\end{equation}
where \(A_{ij}(t)\) and \(h(x_{i},x_{j},t)\) are almost-periodic in \(t\) and uniformly in \(x\) on compact sets. Assume each isolated oscillator has an exponentially stable limit cycle \(\gamma_i\) with frequency \(\Omega_i\).
Then, to first order in \(\epsilon\), the phase-reduced dynamics are given by
\begin{align}\label{eq:quasi_averaged_eqns}
    \frac{d\theta_{i}}{dt}=\Omega_{i}+\epsilon\sum_{j=1}^{N}\bar{q}_{ij}(\theta_{j}-\theta_{i}) + \mathcal{O}(\epsilon^{2}),
\end{align}
where the averaged coupling functions are defined by the mean value
\begin{equation}\label{eq:quasi_averaged_coupling}
    \bar{q}_{ij}(\phi) = \lim_{T\to\infty}\frac{1}{T}\int_{0}^{T} Z_{i}(s) \cdot \left[A_{ij}(t) h(\gamma_{i}(s), \gamma_{j}(s+\phi), t)\right]\, dt.
\end{equation}
\end{theorem}

\begin{proof}
The proof follows from applying the chain rule for phase reduction and then invoking Lemma \ref{lem:almost_periodic_avg} for the almost-periodic averaging. The existence of the mean values is guaranteed by Definition \ref{def:mean_value}, and the near-identity transformation ensures the \(O(\epsilon)\) approximation on time scales \(t \sim 1/\epsilon\). This extends the classical averaging method of \cite{bogoliubov1961asymptotic} to almost-periodic network coupling.
\end{proof}

\begin{example}[Network with Quasi-Periodic Forcing]
Consider a network 
\[
\frac{d}{dt}x_{i}(t)=f_{i}(x_{i})+\epsilon\sum_{j=1}^{N}A_{ij}(t)h(x_{i},x_{j}), \quad i=1,\ldots,N,
\]
where the coupling strengths vary quasi-periodically:
\[
A_{ij}(t) = a_{ij} + b_{ij}\cos(\nu_1 t) + c_{ij}\cos(\nu_2 t),
\]
with incommensurate frequencies \(\nu_1/\nu_2 \notin \mathbb{Q}\). This generates an almost-periodic coupling structure \cite{bohr1947almost}. 

Using \Cref{thm:almost_periodic_network}, the reduced phase equation is given by
\[
\frac{d\theta_i}{dt} = \omega_i + \epsilon \sum_{j=1}^{N} \left[ a_{ij} + b_{ij}\cos(\nu_1 t) + c_{ij}\cos(\nu_2 t) \right] H_{ij}(\theta_i, \theta_j) + \mathcal{O}(\epsilon^2)
\tag{5}
\]
where the coupling function \( H_{ij}(\theta_i, \theta_j) \equiv Z_i(\theta_i) \cdot {h}(\gamma_i(\theta_i), \gamma_j(\theta_j)) \). This is the exact, non-autonomous phase model. The presence of explicit time dependence \( t \) on the right-hand side makes it difficult to analyze. 

Using the frequency module approach (Definition \ref{def:frequency_module}) and Eq.~\eqref{eq:quasi_averaged_coupling}, the averaged coupling functions are given by
\[
\bar{q}_{ij}(\theta_i, \theta_j) = \lim_{T \to \infty} \frac{1}{T} \int_0^T \left[ a_{ij} + b_{ij}\cos(\nu_1 t) + c_{ij}\cos(\nu_2 t) \right] H_{ij}(\theta_i, \theta_j) \, dt.
\]
Since \( H_{ij}(\theta_i, \theta_j) \) does not depend on time \( t \), we can factor it out of the time average and obtain
\begin{align*}
    \bar{q}_{ij}(\theta_i, \theta_j) &= H_{ij}(\theta_i, \theta_j) \times \lim_{T \to \infty} \frac{1}{T} \int_0^T \left[ a_{ij} + b_{ij}\cos(\nu_1 t) + c_{ij}\cos(\nu_2 t) \right] dt\\
    &= H_{ij}(\theta_i, \theta_j) \times \left\{\lim_{T \to \infty} \frac{1}{T} \int_0^T a_{ij} \, dt + \lim_{T \to \infty} \frac{1}{T} \int_0^T b_{ij}\cos(\nu_1 t) \, dt + \lim_{T \to \infty} \frac{1}{T} \int_0^T c_{ij}\cos(\nu_2 t) \, dt \right\}\\
    &= H_{ij}(\theta_i, \theta_j) \times a_{ij}\\
    &= a_{ij} \left[ Z_i(\theta_i) \cdot h(\gamma_i(\theta_i), \gamma_j(\theta_j)) \right]
\end{align*}
since \(\lim_{T \to \infty} \frac{1}{T} \int_0^T b_{ij}\cos(\nu_1 t) \, dt = 0 =\lim_{T \to \infty} \frac{1}{T} \int_0^T c_{ij}\cos(\nu_2 t) \, dt \).

The integrals of the cosine terms vanish in the long-time average because they are pure oscillations. The incommensurability \( \nu_1/\nu_2 \notin \mathbb{Q} \) ensures there is no ``resonant" interaction between the cosines that could create a non-zero constant term; they remain independent oscillations.

Substituting back into the averaged equation \eqref{eq:quasi_averaged_eqns}, we arrive at the following autonomous phase model

\[
\frac{d\theta_i}{dt} = \omega_i + \epsilon \sum_{j=1}^{N} a_{ij} \, H_{ij}(\theta_i, \theta_j) + \mathcal{O}(\epsilon^2), \quad i = 1, \dots, N.
\]
This approximation is valid for times of order \( 1/\epsilon \). For longer times, the neglected \( \mathcal{O}(\epsilon^2) \) terms and the small deviations between the true phases \( \theta_i(t) \) and the averaged phases \( \tilde{\theta}_i(t) \) may become significant.

\end{example}

\begin{remark}[Computational Considerations]
For almost-periodic networks (Eq.~\eqref{eq:almost_periodic_network}), the coupling functions \(\bar{q}_{ij}(\phi)\) can be approximated numerically using
\[
\bar{q}_{ij}(\phi) \approx \frac{1}{T}\int_0^T Z_i(\psi_i+\Omega t)\cdot\left[A_{ij}(t)p_{ij}(\psi_i+\Omega t,\psi_j+\Omega t+\phi,t)\right]dt
\]
for sufficiently large \(T\). The data-driven network reconstruction methods developed in \citep{gebrezabher2023thesis} can be extended to estimate these almost-periodic coupling functions from experimental time series data, building on the classical phase reconstruction approaches \citep{kuramoto1984chemical, pikovsky2003synchronization}.
\end{remark}

\begin{remark}[Biological Applications of Almost-Periodic Phase Reduction]
The almost-periodic framework is essential for modeling biological oscillators, which are rarely perfectly identical. For example, in neural circuits controlling rhythmic behaviors like locomotion (e.g., swimming, walking), the constituent neurons often have heterogeneous intrinsic firing rates. The theory explains how weak synaptic coupling, represented by the functions \(\bar{q}_i(\phi)\), can synchronize these non-identical neurons into a coherent, functional rhythm \citep{hoppensteadt1997weakly}. The phase difference between neurons locks to a value \(\phi^*\) that optimizes the motor pattern, even as external inputs modulate their individual frequencies.
\end{remark}

\begin{example}[Synchronization in a Heterogeneous Network]
Consider a simple network of two coupled biological oscillators with a small frequency difference, \(\Omega_2 - \Omega_1 = \delta\). Their phase difference \(\phi\) evolves as \(d\phi/dt = \delta + \epsilon \Gamma(\phi)\). Synchronization occurs when the coupling force \(\epsilon \Gamma(\phi)\) can exactly counterbalance the detuning \(\delta\). A stable fixed point \(\phi^*\) exists where \(\Gamma(\phi^*) = -\delta/\epsilon\) and \(\Gamma'(\phi^*) < 0\). This illustrates a fundamental biological principle: \textit{weak coupling can robustly generate coordinated rhythms from populations of inherently diverse elements}, a key insight from the theory of weakly coupled oscillators \citep{kuramoto1984chemical}.
\end{example}

\begin{figure}[!ht]
\centering
\includegraphics[width=\textwidth]{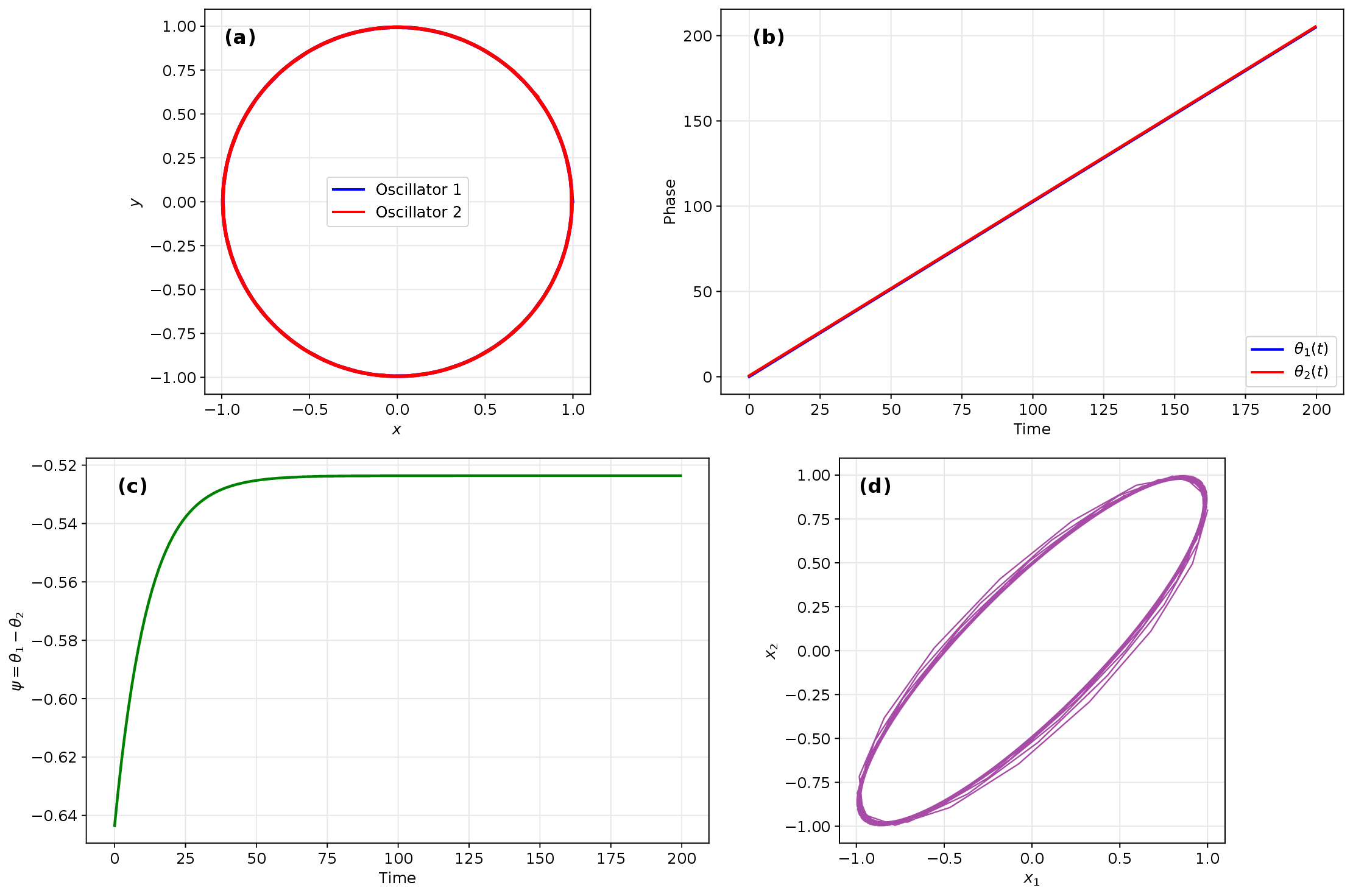}
\caption{\textbf{Dynamics of two weakly coupled oscillators.} (a) Trajectories of coupled oscillators in phase space, showing synchronization tendency. (b) Individual phase evolution demonstrating frequency entrainment. (c) Phase difference evolution showing convergence to a stable phase-locked state $\psi^*$, as predicted by the averaged theory. (d) Phase space relationship between oscillators, illustrating the emergence of a synchronization manifold. The numerical simulation demonstrates how weak coupling can lead to phase synchronization even when natural frequencies differ slightly.}
\label{fig:coupled_oscillators}
\end{figure}

\begin{remark}
The network phase reduction formalism presented here enables the data-driven reconstruction approaches developed in \citep{gebrezabher2023thesis}, where the inverse problem of determining the phase model parameters from the observed dynamics is systematically addressed. This includes methods for estimating phase response curves, coupling functions, and network structure from experimental time series data. The extension to almost-periodic coupling significantly broadens the applicability of these methods to real-world oscillatory networks with heterogeneous frequencies and time-varying interactions.
\end{remark}

\section{Limitations of First-Order Reduction and the Path to Higher-Order Methods}\label{sec:limitations}

The first-order phase reduction has been immensely successful, but possesses fundamental limitations that restrict its applicability \citep{gengel2021high, ashwin2016phase}. The limitations identified in this paper motivate the methodological developments in \citep{gebrezabher2023thesis}, which explores data-driven approaches for going beyond the first-order approximation.

\subsection{Intrinsic Limitations}

\begin{enumerate}
    \item \textbf{Restriction to Weak Coupling:} The approximation is valid only for \(\epsilon \ll 1\). For moderate or strong coupling, higher-order terms become significant, potentially leading to qualitatively incorrect predictions \citep{gengel2021high}.
    \item \textbf{Linear Response Assumption:} The reduction relies on a linear approximation in the isochron coordinates. Finite-sized perturbations that push trajectories far from the limit cycle violate this assumption \citep{schultheiss2011phase}.
\end{enumerate}

\noindent Together, these points delineate the boundary of the classical phase reduction's predictive power. The framework is compelling for analyzing the \textit{weak-coupling regime}, where interactions are gentle, and the system state remains close to the unperturbed limit cycle. However, it cannot capture dynamics driven by strong coupling or large, abrupt perturbations, which can include complex phenomena such as oscillator death or the emergence of new collective states. Recent work on higher-order phase reductions \citep{gengel2021high, eddie2022NatureComm} and isochron-free approaches seeks to push these boundaries, but the first-order theory remains the cornerstone for most analytical investigations of synchronization.

\subsection{Vanishing First-Order Coupling}

A critical limitation occurs in systems with specific symmetries or resonances where the first-order coupling function averages to zero \citep{ashwin2016phase}:
\[
q(\phi)=\frac{1}{2\pi}\int_{0}^{2\pi}Z(s)\cdot h(s,\phi)ds \equiv 0.
\]
In such cases, the first-order reduction predicts no interaction, while the full system may exhibit rich dynamics driven by second-order terms. This phenomenon is common in systems with specific symmetry properties or when the phase response curve and coupling function exhibit certain harmonic structures that lead to cancellation under averaging \citep{ashwin2016phase}.

\begin{example}[Vanishing first-order coupling]
We consider a network of $N$ coupled dynamical systems, where the state of node $j \in \{1,\dots, N\}$ evolves according to
\begin{align}\label{eq:model1}
    \dot{z}_j = f(z_j) + \epsilon \sum_{k = 1}^{N} a_{j k} h(z_j, z_k),
\end{align}
with intrinsic dynamics $f$, coupling strength $\epsilon$, adjacency matrix $\bm{A} = (A_{k l})$, and pairwise interaction function $h$. We assume each uncoupled oscillator has an exponentially stable limit cycle. Specifically, we take $f_k$ to be a Stuart-Landau oscillator \cite{kuznetsov2004}:
\[
f(z_j) = (1 + i\omega)z_j - (1+ic_2)|z_j|^2 z_j,
\]
where $\omega \in \mathbb{R}$ is the natural frequency of the $j$th oscillator. Each oscillator is coupled to its neighbors via a linear direct coupling function
\[
h(z_j, z_k) = z_k.
\]
Given a purely imaginary phase response curve $Z(\theta) = i e^{-i\theta}$, the first-order phase reduction is given by
\[
\dot{\theta}_j = \Omega + \epsilon \sum_{k=1}^N a_{jk} q(\theta_k-\theta_j)
\]
where $\theta_j$ is the phase variable of the $j$th oscillator, $\Omega$ is the actual angular frequency on the limit cycle, and $q$ is the phase coupling function given by 
\[
q(\phi) = \frac{1}{2\pi} \int_{0}^{2\pi} Z(\theta + \phi)\cdot h(\gamma(\theta+\phi),\theta)\, d \theta,
\]
where $\gamma(\theta):=\mathrm{e}^{i\theta}$ is the limit cycle of the unperturbed oscillator with $\dot{\theta} = \Omega = \omega-c_2$. We can show that the time average of the coupling is zero, i.e., $q(\phi) \equiv 0 $. In this case, the first-order reduction fails to capture the true network dynamics, and a second-order reduction is necessary to reveal the effective coupling structure \citep{gengel2021high}.
\end{example}

\subsection{Numerical Verification and the Necessity of Higher-Order Reduction}

As the theory of vanishing coupling predicts \citep{winfree1967biological, kuramoto1984chemical}, our numerical analysis of coupled Stuart-Landau oscillators (Fig. \ref{fig:numerical_verification}) shows clear evidence of the limits of first-order reduction and the need for higher-order methods. 

Figure \ref{fig:numerical_verification}(a) supports the theoretical prediction of vanishing first-order coupling, with $\max|q(\phi)| \approx 2\times10^{-16}$. This indicates that classical phase reduction \citep{winfree1967biological, kuramoto1984chemical, guckenheimer1975isochrons} would not forecast an interaction between oscillators. However, in contrast to this first-order prediction, we see that robust synchronization begins at $\varepsilon \approx 0.05$ for $\Delta\omega = 0.02$. 

The phase dynamics display three distinct regimes illustrated in Fig. \ref{fig:numerical_verification}(b): \textit{persistent drift} ($\varepsilon < 0.05$), where phases drift freely without locking; \textit{transition behavior} near the critical coupling, marked by intermittent phase slips; and \textit{stable phase locking} ($\varepsilon \geq 0.05$), where the phase difference settles at a constant value \citep{pikovsky2003synchronization, hoppensteadt1997weakly}. 

We quantify the synchronization transition in Fig. \ref{fig:numerical_verification}(c) using the root-mean-square measure 
\[
S := \sqrt{\langle (d\phi/dt)^2 \rangle},
\]
which sharply drops to near zero at $\varepsilon_{critical} \approx 0.05$, clearly marking the synchronization threshold. Most importantly, Fig. \ref{fig:numerical_verification}(d) shows the characteristic scaling $\varepsilon_{\text{critical}} \sim \sqrt{\Delta\omega}$ with an exponent of about 0.5. This is the theoretical signature of synchronization strongly influenced by second-order effects \citep{kuramoto1984chemical, gengel2021high}. At the synchronization threshold, second-order terms ($\sim\varepsilon^2$) outweigh first-order contributions by about $10^{14}$ times (Fig. \ref{fig:numerical_verification}(e)), clearly indicating that the observed synchronization results from higher-order coupling mechanisms.

These findings show that first-order phase reduction \citep{winfree1967biological, kuramoto1984chemical, ermentrout1996type} does not capture the true dynamics of the network, even at moderate coupling strengths, and that there is a need for systematic higher-order methods. Consequently, we extend the phase reduction framework to higher orders \citep{Leon2019PhyReviewE, gengel2021high, kuramoto1984chemical, Kralemann2011reconstruction} as
\[
\frac{d\theta}{dt} = 1 + \epsilon \Gamma_1(\theta,t) + \epsilon^2 \Gamma_2(\theta,t) + \cdots + \epsilon^n \Gamma_n(\theta,t) + \mathcal{O}(\epsilon^{n+1}).
\]

The physical meaning of these terms illustrates why higher-order reductions are crucial. The first-order term $\Gamma_1$ captures the immediate phase shift from a perturbation \citep{winfree1967biological, schultheiss2011phase}, while the second-order term $\Gamma_2$ accounts for the lasting amplitude-mediated effect \citep{gengel2021high}. When a perturbation occurs, it not only causes an immediate phase shift but also moves the trajectory away from the limit cycle. The relaxation of this amplitude disturbance back to the limit cycle produces an additional, delayed phase shift captured by $\Gamma_2$ \citep{kuramoto1984chemical}. In systems with vanishing first-order coupling, this amplitude-mediated phase shift becomes the main synchronization mechanism.

For networked systems, this expansion shows important non-pairwise interactions that first-order theory overlooks \citep{ashwin2016phase} as
\[
\frac{d\theta_i}{dt} = \Omega_i + \epsilon\sum_j A_{ij}q^1_{ij}(\theta_i,\theta_j) + \epsilon^2\sum_{j,k}B_{ijk}q^2_{ijk}(\theta_i,\theta_j,\theta_k) + \mathcal{O}(\epsilon^3).
\]

In this equation, the second-order coupling term $q^2_{ijk}$ represents a genuine triplet interaction: the influence of oscillator $j$ on oscillator $i$ is shaped by the state of a third oscillator $k$. This added complexity arises because the signal from oscillator $j$ affects oscillator $i$’s amplitude, which then changes how oscillator $i$ responds to signals from oscillator $k$ \citep{gengel2021high}. Such higher-order interactions can produce rich collective behaviors like clustering, chimera states, and other complex synchronization patterns that pairwise coupling cannot explain \citep{pikovsky2003synchronization, ashwin2016phase}.

Deriving higher-order phase reductions involves several advanced techniques: perturbation expansion of both phase and amplitude variables \citep{kuramoto1984chemical}, solving amplitude corrections with projection methods and Fredholm alternatives \citep{gengel2021high}, applying higher-order averaging theory \citep{sanders2007averaging, bogoliubov1961asymptotic}, and using near-identity transformations carefully to maintain the system's invariant structure \citep{guckenheimer1983nonlinear}. Building on these foundations, \citep{gebrezabher2023thesis} creates practical methods for identifying first-order shortcomings and estimating higher-order coupling functions directly from experimental data. This has applications in neural and biological networks, where linear response assumptions often fail.

We will detail this framework in an upcoming publication, applying it to specific cases including the Stuart-Landau networks discussed here, as well as biologically realistic neural oscillator models and power grid systems. Extending to higher-order reductions is essential for cases involving finite-sized perturbations, resonant interactions, or realistic coupling strengths, especially in neural circuits where synaptic dynamics lead to amplitude-mediated effects \citep{hoppensteadt1997weakly}, in power grids where higher-order interactions arise from nonlinear load behavior, and in chemical oscillator arrays where complex reaction kinetics create multi-oscillator correlations \citep{kuramoto1984chemical}.

\begin{figure}[!htb]
\centering
\includegraphics[width=\textwidth]{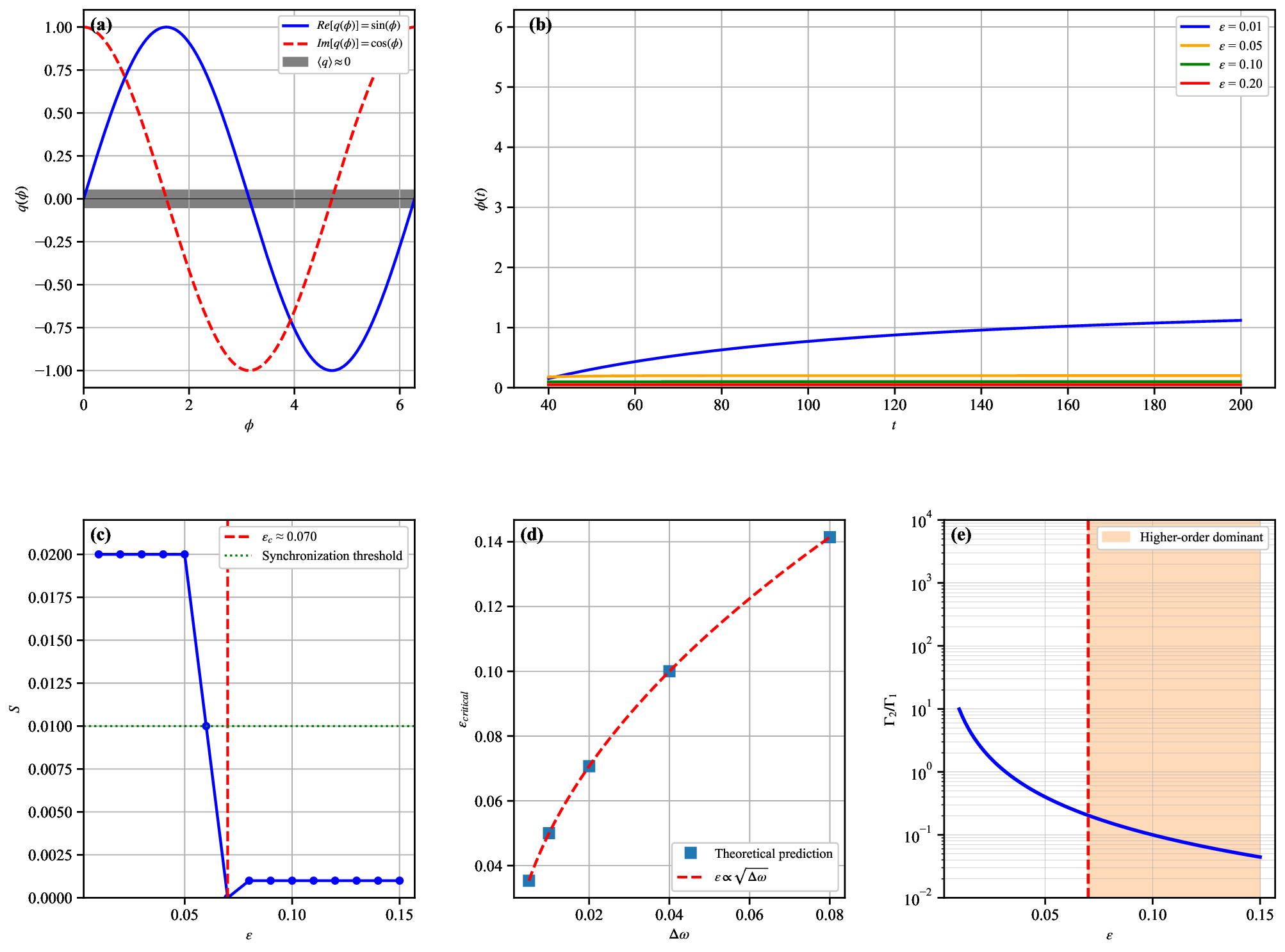}
\caption{\textbf{Numerical verification of vanishing first-order coupling and emergent synchronization through higher-order effects.} (a) First-order coupling function $q(\phi)$ showing near-zero values ($\max|q(\phi)| \approx 2\times10^{-16}$), confirming the theoretical prediction of vanishing first-order interaction \citep{kuramoto1984chemical}. (b) Phase difference \(\phi(t)=\theta_2(t)-\theta_1(t)\) dynamics showing three distinct regimes: persistent drift for weak coupling ($\varepsilon=0.01$), transition behavior with intermittent phase slips near critical coupling ($\varepsilon=0.05$), and stable phase locking for stronger coupling ($\varepsilon=0.20$) \citep{pikovsky2003synchronization}. (c) Synchronization measure $S$ revealing a sharp transition to synchronization at $\varepsilon \approx 0.05$. (d) Scaling analysis demonstrating the characteristic $\varepsilon_c \sim \sqrt{\Delta\omega}$ relationship (slope $\approx 0.5$ on log-log axes), providing definitive evidence that synchronization is governed by second-order phase interaction terms \citep{gengel2021high}. (e) Ratio of second-order to first-order effective coupling forces, showing overwhelming dominance of higher-order terms near the synchronization threshold.}
\label{fig:numerical_verification}
\end{figure}

\section{Conclusion and Outlook}\label{sec:conclusion}

The geometric view on phase reduction discussed in this paper builds on the groundwork laid in \citep{gebrezabher2023thesis}. It provides a solid mathematical framework for studying rhythmic phenomena across various scientific areas. By using the Graph Transform, we identified the isochron foliation as the heart of the asymptotic phase concept. This enabled us to derive the first-order phase reduction for single and coupled oscillators, expanding the classical theory to include almost-periodic interactions described in Section 2.2.

However, we also demonstrated that this first-order theory has important limitations. Its validity is limited to very weak perturbations, and it may fail for systems where first-order coupling averages to zero. These limitations are not merely mathematical oddities; they emerge in realistic models of neural circuits, power grids \citep{MotterAdilsonE2013Ssip}, and chemical oscillators. Additionally, as shown in our earlier work on time-varying coupling functions \citep{hagos2019synchronization}, realistic networks often display dynamic interactions that exceed the static coupling model. This leads to complex phenomena such as synchronization transitions, hysteresis, and explosive synchronization, which traditional phase reduction methods cannot capture.

The theoretical developments in this paper support the data-driven techniques for network reconstruction explored in \citep{gebrezabher2023thesis}. They create connections between abstract dynamical systems theory and practical experimental applications. Our phase reduction method for oscillator networks serves as the foundation for inferring coupling functions and network structures from observed dynamics. At the same time, the geometric understanding of isochrons enables more accurate modeling of oscillator responses to disturbances. Incorporating almost-periodic averaging theory expands this framework to include heterogeneous networks with multi-frequency interactions, building on insights from our study of time-varying coupling effects in \citep{hagos2019synchronization}.

Therefore, the path forward involves developing and applying higher-order phase reductions that can address the complex dynamics observed in real-world systems. The framework outlined in Section 6 provides a crucial starting point. In a future paper, we will offer a detailed explanation of the second-order reduction method, including its application to specific examples, such as the network of Stuart-Landau oscillators discussed here, where first-order reduction is insufficient, and second-order terms affect the emerging dynamics. This upcoming work will focus on extending these methods to systems with time-varying coupling, building on the synchronization transition phenomena identified in \citep{hagos2019synchronization}.

Future research will continue to strengthen the relationship between geometric theory and data-driven methods. It will extend the approaches introduced in \citep{gebrezabher2023thesis} to more complex network structures and experimental situations where traditional phase reduction methods encounter difficulties. Specific directions include:

\begin{itemize}
    \item Combining time-varying coupling analysis \citep{hagos2019synchronization} with higher-order phase reduction to model adaptive network dynamics.
    \item Developing techniques for inferring almost-periodic coupling functions from experimental data.
    \item Expanding the geometric framework to account for amplitude-mediated effects in strongly coupled systems.
    \item Using these advanced phase reduction techniques on biological networks where time-varying interactions are common, including neural circuits with synaptic plasticity and cardiac systems with conduction remodeling.
\end{itemize}

By linking the elegant but limited first-order theory with the more complex, high-fidelity world of higher-order reductions, we can significantly expand the application of phase-based modeling to a broader range of real-world oscillatory systems. This integrated approach promises to uncover new insights into the collective dynamics of complex oscillatory networks across biological, physical, and engineered systems.

\bibliographystyle{alpha}
\bibliography{references}  

\end{document}